\documentclass[a4paper,11pt,reqno]{amsart}
\usepackage{amsmath}
\usepackage{amsthm,enumerate}

\usepackage[pagewise]{lineno}

\usepackage{graphicx}
\usepackage{amssymb}

\usepackage{appendix}

\usepackage[italian,english]{babel}

\usepackage[utf8x]{inputenc}
\usepackage{fancyhdr}

\usepackage{calc}
\usepackage{url}

\DeclareMathOperator{\dive}{div}
\DeclareMathOperator{\s}{\sigma}

\usepackage[text={6.3in,8.6in},centering]{geometry}

\usepackage{srcltx, inputenc}

\usepackage[T1]{fontenc}

\usepackage[usenames,dvipsnames]{color}

\makeatletter
\def\cleardoublepage{\clearpage\if@twoside \ifodd\c@page\else%
         \hbox{}%
     \thispagestyle{empty}
     \newpage%
     \if@twocolumn\hbox{}\newpage\fi\fi\fi}
\makeatother

\theoremstyle{plain}
\newtheorem{theorem}{Theorem}[section]

\newtheorem{definition}[theorem]{Definition}

\newtheorem{lemma}[theorem]{Lemma}

\newtheorem{proposition}[theorem]{Proposition}
\newtheorem{remark}[theorem]{Remark}

\numberwithin{equation}{section}
\theoremstyle{definition}

\newcommand{\R}{\ensuremath{\mathbb{R}}}

\begin{document}

\title[]{Global existence for reaction-diffusion evolution equations driven by the $\text{p}$-Laplacian on manifolds}
\author{Gabriele Grillo}
\address{\hbox{\parbox{5.7in}{\medskip\noindent{Dipartimento di Matematica,\\
Politecnico di Milano,\\
   Piazza Leonardo da Vinci 32, 20133 Milano, Italy.
   \\[3pt]
        \em{E-mail address: }{\tt
          gabriele.grillo@polimi.it
          }}}}}

\author{Giulia Meglioli}
\address{\hbox{\parbox{5.7in}{\medskip\noindent{Dipartimento di Matematica,\\
Politecnico di Milano,\\
   Piazza Leonardo da Vinci 32, 20133 Milano, Italy.
   \\[3pt]
        \em{E-mail address: }{\tt
          giulia.meglioli@polimi.it
          }}}}}

\author{Fabio Punzo}
\address{\hbox{\parbox{5.7in}{\medskip\noindent{Dipartimento di Matematica,\\
Politecnico di Milano,\\
   Piazza Leonardo da Vinci 32, 20133 Milano, Italy. \\[3pt]
        \em{E-mail address: }{\tt
          fabio.punzo@polimi.it}}}}}


\maketitle

\maketitle              

\begin{abstract}
We consider reaction-diffusion equations driven by the $p$-Laplacian on noncompact, infinite volume manifolds assumed to support the Sobolev inequality and, in some cases, to have $L^2$ spectrum bounded away from zero, the main example we have in mind being the hyperbolic space of any dimension. It is shown that, under appropriate conditions on the parameters involved and smallness conditions on the initial data, global in time solutions exist and suitable smoothing effects, namely explicit bounds on the $L^\infty$ norm of solutions at all positive times, in terms of $L^q$ norms of the data. The geometric setting discussed here requires significant modifications w.r.t. the Euclidean strategies.
\end{abstract}

\section{Introduction}
We investigate existence of nonnegative global in time solutions to the quasilinear parabolic problem
\begin{equation}\label{problema}
\begin{cases}
 u_t = \operatorname{div} \left(|\nabla u|^{p-2}\nabla u \right)+\, u^{\s} & \text{in}\,\, M \times (0,T) \\
u =u_0 &\text{in}\,\, M\times \{0\}\,,
\end{cases}
\end{equation}
where $M$ is an $N$-dimensional, complete, noncompact, Riemannian manifold of infinite volume, whose metric is indicated by $g$, and where $\dive$ and $\nabla$ are respectively the divergence and the gradient with respect to $g$ and $T\in(0,+\infty]$. We shall assume throughout this paper that \begin{equation}\label{costanti}
     \frac{2N}{N+1}<p<N, \quad\quad \s>p-1.\end{equation}
The problem is posed in the Lebesgue spaces
$$L^q(M)=\left\{v:M\to\R\,\, \text{measurable}\,\,  ,   \,\, \|v\|_{L^q}:=\left(\int_{M} v^q\,d\mu\right)^{1/q}<+\infty\right\},$$
where $\mu$ is the Riemannian measure on $M$. We also assume the validity of the Sobolev inequality:
\begin{equation}\label{S}
\text{(Sobolev inequality)}\quad\|v\|_{L^{p^*}(M)} \le \frac{1}{C_{s,p}} \|\nabla v\|_{L^p(M)}\quad \text{for any}\,\,\, v\in C_c^{\infty}(M),
\end{equation}
where $C_{s,p}>0$ is a constant and $p^*:=\frac{pN}{N-p}$. In some cases we also assume that the Poincaré inequality is valid, that is
\begin{equation}\label{P}
\text{(Poincar\'e inequality)}\quad \|v\|_{L^p(M)} \le \frac{1}{C_p} \| \nabla v\|_{L^p(M)} \quad \text{for any}\,\,\, v\in C_c^{\infty}(M),
\end{equation}
for some $C_p>0$. Observe that, for instance, \eqref{S} holds if $M$ is a Cartan-Hadamard manifold, i.e. a simply connected Riemannian manifold with nonpositive sectional curvatures, while \eqref{P} is valid when $M$ is a Cartan-Hadamard manifold satisfying the additional condition of having sectional curvatures bounded above by a constant $−c < 0$ (see, e.g. \cite{Grig1,Grig2}). Therefore, as it is well known, on $\R^N$ \eqref{S} holds, but \eqref{P} fails, whereas on the hyperbolic space both \eqref{S} and \eqref{P} are fulfilled.

\medskip

Global existence and finite time blow-up of solutions for problem \eqref{problema} has been deeply studied when $M=\R^N$, especially in the case $p=2$ (linear diffusion). The literature for this problem is huge and there is no hope to give a comprehensive review here. We just mention the fundamental result of Fujita, see \cite{F}, who shows that blow-up in a finite time occurs for all nontrivial nonnegative data when $\s<1+\frac2N$, while global existence holds, for $\s>1+ \frac 2N$, provided the initial datum is small enough in a suitable sense. Furthermore, the critical exponent $\s=1+\frac 2N$, belongs to the case of finite time blow-up, see e.g. \cite{H} for the one dimensional case, $N=1$, or \cite{KST} for $N>1$. For further results concerning problem \eqref{problema} with $p=2$ see e.g. \cite{CFG,DL,FI,L,Q,Pu3,S,Sun1,W,Y,Z}).
\small

Similarly, the case of problem \eqref{problema} when $M=\R^N$ and $p>1$ has attracted a lot of attention, see e.g. \cite{G1, G2, G3, Mitidieri, Mitidieri2, Poho, PT} and references therein. In particular, in \cite{Mitidieri2}, nonexistence of nontrivial weak solutions is proved for problem \eqref{problema} with $M=\R^N$ and
$$
p>\frac{2N}{N+1},\quad\quad \max\{1,p-1\}<\s\le p-1+\frac p N\,.
$$
Similar \it weighted \rm problems have also been treated. In fact, for any strictly positive measurable function $\rho:\R^N\to\R$, let us consider the weighted ${\textrm L}^{q}_{\rho}$ spaces
$${\textrm L}^q_{\rho}(\R^N)=\left\{v:\R^N\to\R\,\, \text{measurable}\,\,  ,   \,\, \|v\|_{L^q_{\rho}}:=\left(\int_{\R^N} v^q\rho(x)\,dx\right)^{1/q}<+\infty\right\}.$$
In \cite{MT1} problem
\begin{equation}\label{AA}
\begin{cases}
\rho(x) u_t = \operatorname{div} \left(|\nabla u|^{p-2}\nabla u \right)+\, \rho(x)u^{\s} & \text{in}\,\, \R^N \times (0,T) \\
u =u_0 &\text{in}\,\, \R^N\times \{0\}\,,
\end{cases}
\end{equation}
is addressed. In \cite[Theorem 1]{MT1}, it is showed that, when $p>2$, $\rho(x)=(1+|x|)^{-l}$, $0\le l<p$, $\s>p-1+\frac pN$, $ u_0\in {\textrm L}^1_{\rho}(\R^N)\cap {\textrm L}^s_{\rho}(\R^N)$ is sufficiently small, with $s>\frac{(N-l)(\s-p+1)}{p-l}$, then problem \eqref{AA} admits a global in time solution. Moreover, the solution satisfies a smoothing estimate ${\textrm L}^1_{\rho}-{\textrm L}^{\infty}$, in the sense that for sufficiently small data $ u_0\in {\textrm L}^1_{\rho}(\R^N)$, the corresponding solution is bounded, and a quantitive bound on the ${\textrm L}^{\infty}$ norm of the solution holds, in term of the ${\textrm L}^1_{\rho}(\R^N)$ norm of the initial datum. On the other hand, in \cite[Theorem 2]{MT1}, when $p>2$, $\rho(x)=(1+|x|)^{-l}$, $l\ge p$, $\s>p-1$, $ u_0\in {\textrm L}^1_{\rho}(\R^N)\cap {\textrm L}^s_{\rho}(\R^N)$ is sufficiently small, with $s>\frac{N}{p}(\s-p+1)$, then problem \eqref{AA} admits a global in time solution, which is bounded for positive times.
\medskip

On the other hand, existence and nonexistence of global in time solutions to problems closely related to problem \eqref{problema} have been investigated also in the Riemannian setting. The situation can be significantly different from the Euclidean situation, especially in the case of negative curvature. Infact, when dealing the the case of the $N$-dimensional \it hyperbolic space\rm, $M=\mathbb{H}^N$, it is known that when $p=2$, for all $\s>1$ and sufficiently small nonnegative data there exists a global in time solution, see \cite{BPT}, \cite{WY}, \cite{WY2}, \cite{Pu3}. A similar result has been also obtain when $M$ is a complete, noncompact, stochastically complete Riemannian manifolds with $\lambda_1(M)>0$, where $\lambda_1(M):=\inf \operatorname{spec}(-\Delta)$, see \cite{GMP3}. Stochastic completeness amounts to requiring that the linear heat semigroup preserves the identity, and is known to hold e.g. if the sectional curvature satisfies $\text{sec}(x)\,\ge -c d(x,o)^2$ for all $x\in M$ outside a given compact, and a suitable $c>0$, where $d$ is the Riemannian distance and $o$ is a given pole. Besides, it is well known that $\lambda_1(M)>0$ e.g. if $\text{sec}(x)\,\le -c<0$ for all $x\in M$. Therefore, the class of manifolds for which the results of \cite{GMP3} hold is large, since it includes e.g. all Cartan-Hadamard manifolds with curvature bounded away from zero and not diverging faster than quadratically at infinity.
\smallskip

Concerning problem \eqref{problema} with $p>1$, we refer the reader to \cite{MMP, MeMoPu} and references therein. In particular, in \cite{MMP}, nonexistence of global in time solutions on infinite volume Riemannian manifolds $M$ is shown under suitable weighted volume growth conditions. In \cite{MeMoPu}, problem \eqref{problema} with $M=\Omega$ being a bounded domain and $u^{\s}$ replaced by $V(x,t)u^{\s}$ is addressed, where $V$ is a positive potential. To be specific, nonexistence of nonnegative, global solutions is established under suitable integral conditions involving $V$, $p$ and $\s$.
\bigskip

In this paper, we prove the following results. Assume that the bounds \eqref{costanti} and the Sobolev inequality \eqref{S} hold, and besides that $\s>p-1+\frac pN$.
\begin{itemize}
\item[(a)] If $u_0\in {\textrm L}^{s}(M)\cap {\textrm L}^1(M)$ is sufficiently small, with $s>(\s-p+1)\frac Np$, then a global solution exists. Furthermore, a smoothing estimate of the type ${\textrm L}^1-{\textrm L}^{\infty}$ holds (see Theorem \ref{teo1}).
\item[(b)] If $u_0\in {\textrm L}^{(\s-p+1)\frac Np}(M)$ is sufficiently small, then a global solution exists. Furthermore, a smoothing estimate of the type ${\textrm L}^{(\s-p+1)\frac Np}-{\textrm L}^{\infty}$ holds (see Theorem \ref{teo11}), this being new even in the Euclidean case.
\item[(c)] In addition, in both the latter two cases, we establish a ${\textrm L}^{(\s-p+1)\frac Np}-{\textrm L}^q$ smoothing estimate, for any $(\s-p+1)\frac Np\le q<+\infty$ and an ${\textrm L}^q-{\textrm L}^q$ estimate for any $1<q<+\infty$, for suitable initial data $u_0$.
\end{itemize}

Now suppose that both the Sobolev inequality \eqref{S} and the Poincar\'e inequality \eqref{P} hold, and that \eqref{costanti} holds. This situation has of course no Euclidean analogue, as it is completely different from the case of a bounded Euclidean domain since $M$ is noncompact and of infinite measure.
Then:
\begin{itemize}
\item[(d)] If $u_0\in {\textrm L}^{s}(M)\cap {\textrm L}^{\s\frac Np}(M)$ is sufficiently small, with $s>\max\left\{(\s-p+1)\frac Np, 1\right\}$, then a global solution exists. Furthermore, a smoothing estimate of the type ${\textrm L}^s-{\textrm L}^{\infty}$ holds (see Theorem \ref{teo71}).
\item[(e)] In addition, we establish and ${\textrm L}^{\s\frac Np}-{\textrm L}^q$ estimate, for any $\s\frac Np\le q<+\infty$ and an ${\textrm L}^q-{\textrm L}^q$ estimate for any $1<q<+\infty$, for suitable initial data $u_0$.
\end{itemize}
Note that, when we require both \eqref{S} and \eqref{P}, the assumption on $\sigma$ can be relaxed.
\smallskip

In order to prove (a), we adapt the methods exploited in \cite[Theorem 1]{MT1}. Moreover, (b), (c) and (e) are obtained by means of an appropriate use of the Moser iteration technique, see also \cite{GMP2} for a similar result in the case of the porous medium equation with reaction. The proof of statement (d) is inspired \cite[Theorem 2]{MT1}; however, significant changes are needed since in \cite{MT1} the precise form of the weight $\rho$ is used.
\bigskip

The paper is organized as follows. The main results are stated in Section \ref{prel}. Section \ref{Lq} is devoted to ${\textrm L}^{q_0}-{\textrm L}^q$ and ${\textrm L}^q-{\textrm L}^q$ smoothing estimates, mainly instrumental to what follows. Some a priori estimates are obtained in Section \ref{aux}. In Sections \ref{dim1}, \ref{dim2} and \ref{dim3}, Theorems \ref{teo1}, \ref{teo11} and \ref{teo71} are proved, respectively. Finally, in Section \ref{porous} we state similar results for the porous medium equation with reaction; the proofs are omitted since they are entirely similar to the p-Laplacian case.

\section{Statements of the main results}\label{prel}

Solutions to \eqref{problema} will be meant in the weak sense, according to the following definition.

\begin{definition}\label{def21}
Let $M$ be a complete noncompact Riemannian manifold of infinite volume. Let $p>1$, $\s>p-1$ and $u_0\in{\textrm L}^{1}_{\textit{loc}}(M)$, $u_0\ge0$.
We say that the function $u$ is a weak solution to problem \eqref{problema} in the time interval $[0,T)$ if
$$
u\in {\textrm L}^2((0,T);{\textrm W}^{1,p}_{loc}(M)) \cap {\textrm L}^{\s}_{loc}(M\times(0,T)) 
$$
and for any $\varphi \in C_c^{\infty}(M\times[0,T])$ such that $\varphi(x,T)=0$ for any $x\in M$, $u$ satisfies the equality:
\begin{equation*}
\begin{aligned}
-\int_0^T\int_{M} \,u\,\varphi_t\,d\mu\,dt =&\int_0^T\int_{M} |\nabla u|^{p-2}\left \langle \nabla u, \nabla \varphi \right \rangle\,d\mu\,dt\,+ \int_0^T\int_{M} \,u^{\s}\,\varphi\,d\mu\,dt \\
& +\int_{M} \,u_0(x)\,\varphi(x,0)\,d\mu.
\end{aligned}
\end{equation*}
\end{definition}

First we consider the case that $\s>p-1+\frac p N$ and that the Sobolev inequality holds on $M$. In order to state our results, we define
\begin{equation}\label{sig0}\s_0:=(\s-p+1)\frac{N}{p}.\end{equation} Observe that $\s_0>1$ whenever $\s>p-1+\frac pN$. Our first result is a generalization of \cite{MT1} to the geometric setting considered here.

\begin{theorem}\label{teo1}
Let $M$ be a complete, noncompact, Riemannian manifold of infinite volume such that the Sobolev inequality \eqref{S} holds. Assume \eqref{costanti} holds and, besides, that $\s>p-1+\frac pN$, $s>\s_0$ and $u_0\in{\textrm L}^{s}(M)\cap L^1(M)$, $u_0\ge0$ where $\s_0$ has been defined in \eqref{sig0}.

\begin{itemize}
\item[(i)] Assume that
\begin{equation}\label{epsilon0}
\|u_0\|_{\textrm L^{s}(M)}\,<\,\varepsilon_0,\quad  \|u_0\|_{\textrm L^{1}(M)}<\,\varepsilon_0\,,
\end{equation}
with $\varepsilon_0=\varepsilon_0(\s,p,N, C_{s,p})>0$ sufficiently small. Then problem \eqref{problema} admits a solution for any $T>0$,  in the sense of Definition \ref{def21}. Moreover, for any $\tau>0,$ one has $u\in L^{\infty}(M\times(\tau,+\infty))$ and there exists a constant $\Gamma>0$ such that, one has
\begin{equation}\label{eq21tot}
\|u(t)\|_{L^{\infty}(M)}\le \Gamma\, t^{-\alpha}\,\|u_0\|_{L^{1}(M)}^{\frac{p}{N(p-2)+p}}\,\quad \text{for all $t>0$,}
\end{equation}
where
\begin{equation*}
\alpha:=\frac{N}{N(p-2)+p}\,.
\end{equation*}
\item[(ii)] Let  $\s_0\le q<\infty$. If
\begin{equation}\label{eps3a}
\|u_0\|_{L^{\s_0}(M)}< \hat \varepsilon_0
\end{equation}
for $\hat\varepsilon_0=\hat\varepsilon_0(\s, p , N, C_{s,p}, q)>0$ small enough, then there exists a constant $C=C(\s,p,N,\varepsilon_0,C_{s,p}, q)>0$ such that
\begin{equation}\label{eq23}
\|u(t)\|_{L^q(M)}\le C\,t^{-\gamma_q} \|u_{0}\|^{\delta_q}_{L^{\s_0}(M)}\quad \textrm{for all }\,\, t>0\,,
\end{equation}
where
$$
\gamma_q=\frac{1}{\s-1}\left[1-\frac{N(\s-p+1)}{pq}\right],\quad \delta_q=\frac{\s-p+1}{\s-1}\left[1+\frac{N(p-2)}{pq}\right]\,.
$$
 \item[(iii)]Finally, for any $1\le q<\infty$, if $u_0\in {\textrm L}^q(M)\cap\textrm L^{\s_0}(M)$ and
\begin{equation}\label{eps2}
\|u_0\|_{\textrm L^{\s_0}(M)}\,<\,\varepsilon
\end{equation}
with $\varepsilon=\varepsilon(\s,p,N, C_{s,p},q)>0$ sufficiently small, then
\begin{equation}\label{eq22}
\|u(t)\|_{L^q(M)}\le  \|u_{0}\|_{L^q(M)}\quad \textrm{for all }\,\, t>0\,.
\end{equation}
\end{itemize}
\end{theorem}

\begin{remark}\label{rem2}\rm
Observe that the choice of $\varepsilon_0$ in \eqref{epsilon0} is made in Lemma \ref{lemma4}. Moreover, the proof of the above theorem will show that one can take an explicit value of $\hat \varepsilon_0$ in \eqref{eps3a} and $\varepsilon$ in \eqref{eps2}. In fact, let $q_0>1$ be fixed and $\{q_n\}_{n\in\mathbb{N}}$ be the sequence defined by: $$q_n=\frac{N}{N-p}(p+q_{n-1}-2), \quad \text{for all}\,\,n\in\mathbb{N},$$
so that
\begin{equation}\label{eq400}
q_n=\left(\frac{N}{N-p}\right)^{n}q_0+\frac{N}{N-p}(p-2) \sum_{i=0}^{n-1} \left(\frac{N}{N-p}\right)^i.
\end{equation}
Clearly, $\{q_n\}$ is increasing and $q_n \longrightarrow +\infty$ as $n\to +\infty.$ Fix $q\in[q_0,+\infty)$ and let $\bar n$ be the first index such that $q_{\bar n}\ge q$. Define
\begin{equation}\label{eq40}
\tilde\varepsilon_0=\tilde\varepsilon_0(\s,p,N,C_{s,p},q,q_0):=\left[\min \left\{\min_{n=0,...,\bar n}\left(\frac{p( q_n-1)^{1/p}}{p+q_n-2}\right)^p;\left(\frac{p(\s_0-1)^{1/p}}{p-\s_0-2}\right)^p\right\}\frac{C_{s,p}^p}{2}\right]^{\frac{1}{\s-p+1}}.
\end{equation}
Observe that $\tilde\varepsilon_0$ in \eqref{eq40} depends on the value $q$ through the sequence $\{q_n\}$. More precisely, $\bar n$ is increasing with respect to $q$, while the quantity $\min_{n=0,...,\bar n}( q_n-1)\left(\frac{p}{p+q_n-2}\right)^p\frac{C_{s,p}^p}{2}$ decreases w.r.t. $q$. 

\noindent Then, in \eqref{eps3a} we can take
\begin{equation*}\label{epsilon0aaa}
\hat \varepsilon_0=\hat\varepsilon_0(\s, p, N, C_{s,p}, q)=\tilde\varepsilon_0(\s, p, N, C_{s,p}, q, \s_0)\,.
\end{equation*}

\noindent Similarly, in \eqref{eps2}, we can take
\begin{equation}\label{epsilon000}
\varepsilon= \bar\varepsilon_0\wedge \hat\varepsilon_0,
\end{equation}
where
\[\bar\varepsilon_0=\bar\varepsilon_0(\s, p, C_{s,p}, q):=\left[\min \left\{\left(\frac{p( q-1)^{1/p}}{p+q-2}\right)^pC_{s,p}^p;\,\,\left(\frac{p(\s_0-1)^{1/p}}{p-\s_0-2}\right)^pC_{s,p}^p\right\}\right]^{\frac1{\s-p+1}}\,.\]
\end{remark}
\smallskip

The next result involves a similar smoothing effect for a different class of data. Such result seems to be new also in the Euclidean setting.

\begin{theorem}\label{teo11}
Let $M$ be a complete, noncompact, Riemannian manifold of infinite volume such that the Sobolev inequality \eqref{S} holds. Assume \eqref{costanti} and, besides, that $\s>p-1+\frac pN$ and $u_0\in{\textrm L}^{\s_0}(M)$, $u_0\ge0$, with $\s_0$ as in \eqref{sig0}. Assume that
\begin{equation}\label{epsilon000}
\|u_0\|_{\textrm L^{\s_0}(M)}\,<\,\varepsilon_2\,,
\end{equation}
with $\varepsilon_2=\varepsilon_2(\s,p,N, C_{s,p},q)>0$ sufficiently small. Then problem \eqref{problema} admits a solution for any $T>0$,  in the sense of Definition \ref{def21}. Moreover, for any $\tau>0,$ one has $u\in L^{\infty}(M\times(\tau,+\infty))$ and for any $q>\s_0$, there exists a constant $\Gamma>0$ such that, one has
\begin{equation}\label{eq21totold}
\|u(t)\|_{L^{\infty}(M)}\le \Gamma\, t^{-\frac1{\s-1}}\,\|u_0\|_{L^{\s_0}(M)}^{\frac{\s-p+1}{\s-1}}\, \quad \text{for all $t>0$.}
\end{equation}
Moreover, (ii) and (iii) of Theorem \ref{teo1} hold.
\end{theorem}

\begin{remark}
We comment that, as in Remark \ref{rem2}, one can choose an explicit value for $\varepsilon_2$ in \eqref{epsilon000}. In fact, let $q_0=\s_0$ in \eqref{eq40}. It can be shown that one can take, with this choice of $q_0$:
$$
\varepsilon_2=\varepsilon_2(\s,p,N, C_{s,p},q):=\min\left\{\tilde \varepsilon_0(\s, p, N, C_{s,p}, q, \s_0)\,;\left(\frac{1}{C\,\tilde C}\right)^{\frac 1{\s-p+1}}\right\}\,,
$$
where $C>0$ and $\tilde C>0$ are defined in Proposition \ref{prop42} and Lemma \ref{lemma3}, respectively.
\end{remark}

\begin{remark}
Observe that, due to the assumption $\s>p-1+\frac pN$, one has
$$
\frac 1{\s-1}<\frac{N}{N(p-2)+p}.
$$
Hence, for large times, the decay given by Theorem \ref{teo11} is worse than the one of Theorem \ref{teo1}; however, in this regards, note that the assumptions on the initial datum $u_0$ are different in the two theorems. On the other hand, estimates \eqref{eq21totold} and \eqref{eq21tot}, are not sharp in general for small times. For example, when $u_0\in L^{\infty}(M)$, $u(t)$ remains bounded for any $t\in[0,T)$, where $T$ is the maximal existence time.
\end{remark}

\medskip

In the next theorem, we address the case $\s>p-1$, assuming that both the inequalities \eqref{S} and \eqref{P} hold on $M$, hence with stronger assumptions on the manifold considered. This has of course no Euclidean analogue, as the noncompactness of the manifold considered, as well as the fact that it has infinite volume, makes the situation not comparable to the case of a bounded Euclidean domain.

\begin{theorem}\label{teo71}
Let $M$ be a complete, noncompact manifold of infinite volume such that the Sobolev inequality \eqref{S} and the Poincaré inequality \eqref{P} hold. Assume that \eqref{costanti} holds, and besides that $p>2$. Let $u_0\ge0$ be such that $u_0\in{\textrm L}^{s}(M)\cap {\textrm L}^{\s\frac Np}(M)$, for some $s>\max\left\{\s_0,1\right\}$ and $q_0>1$. Assume also that
\begin{equation*}\label{epsilon1}
\|u_0\|_{\textrm L^{s}(M)}\,<\,\varepsilon_1, \quad \|u_0\|_{\textrm L^{\s\frac Np}(M)}\,<\,\varepsilon_1,
\end{equation*}
with $\varepsilon_1=\varepsilon_1(\s,p,N, C_{s,p}, C_p,s)$ sufficiently small. Then problem \eqref{problema} admits a solution for any $T>0$,  in the sense of Definition \ref{def21}. Moreover, for any $\tau>0,$ one has $u\in L^{\infty}(M\times(\tau,+\infty))$ and, for any $q>s$, there exists a constant $\Gamma>0$ such that, one has
\begin{equation}\label{eq801}
\|u(t)\|_{L^{\infty}(B_R)}\le \Gamma\, t^{-\beta_{q,s}}\,\|u_0\|_{L^{s}(B_R)}^{\frac{ps}{N(p-2)+pq}}\,\quad \text{for all $t>0$,}
\end{equation}
where
\begin{equation*}\label{beta}
\beta_{q,s}:=\frac{1}{p-2}\left(1-\frac{ps}{N(p-2)+pq}\right)>0\,.
\end{equation*}
Moreover, let  $s\le q<\infty$ 
and
\begin{equation*}\label{eps3b}
\|u_0\|_{L^{s}(M)}< \hat \varepsilon_1
\end{equation*}
for $\hat\varepsilon_1=\hat\varepsilon_1(\s, p , N, C_{s,p},C_p, q,s)$ small enough. Then there exists a constant $C=C(\s,p,N,\varepsilon_1,C_{s,p},C_p, q,s)>0$ such that
\begin{equation}\label{eq23b}
\|u(t)\|_{L^q(M)}\le C\,t^{-\gamma_q} \|u_{0}\|^{\delta_q}_{L^{s}(M)}\quad \textrm{for all }\,\, t>0\,,
\end{equation}
where
$$
\gamma_q=\frac{s}{p-2}\left[\frac 1s-\frac 1q\right],\quad \delta_q=\frac sq\,.
$$
Finally, for any $1<q<\infty$, if $u_0\in {\textrm L}^q(M)\cap\textrm L^{s}(M)$ and
\begin{equation*}\label{eps2b}
\|u_0\|_{\textrm L^{s}(M)}\,<\,\varepsilon
\end{equation*}
with $\varepsilon=\varepsilon(\s,p,N, C_{s,p},C_p,q)$ sufficiently small, then
\begin{equation}\label{eq22b}
\|u(t)\|_{L^q(M)}\le  \|u_{0}\|_{L^q(M)}\quad \textrm{for all }\,\, t>0\,.
\end{equation}
\end{theorem}

\begin{remark}\label{rem1}
It is again possible to give an explicit estimate on the smallness parameter $\varepsilon_1$ above. In fact, let $q_0>1$ be fixed and $\{q_m\}_{m\in\mathbb{N}}$ be the sequence defined by: $$q_m=p+q_{m-1}-2, \quad \text{for all}\,\,m\in\mathbb{N},$$
so that
\begin{equation}\label{eq70}
q_m=q_0+m(p-2).
\end{equation}
Clearly, $\{q_m\}$ is increasing and $q_m \longrightarrow +\infty$ as $m\to +\infty.$ Fix $q\in[q_0,+\infty)$ and let $\bar m$ be the first index such that $q_{\bar m}\ge q$. Define $\tilde\varepsilon_1=\tilde\varepsilon_1(\s,p,N,C_{s,p},C_p,q,q_0)$ such that
\begin{equation*}\label{eq71}
\begin{aligned}
\tilde\varepsilon_1:=\min &\left\{\left[\min_{m=0,...,\bar m}\left(\frac{p( q_m-1)^{1/p}}{p+q_m-2}\right)^pC\right]^{\frac{\s+p+q_m-2}{\s(\s+q_m-1)-p(p+q_m-2)}};\right.\\
&\quad\left.\left[\left(\frac{p( \s\frac N{p}-1)^{1/p}}{(p+\s\frac{N}{p}-2)}\right)^pC\right]^{\frac{\s+p+\s\frac N{p}-2}{\s(\s+\s\frac N{p}-1)-p(p+\s\frac N{p}-2)}}\right\}
\end{aligned}
\end{equation*}
where $C=\tilde C\,C_p^{p(\frac{p-1}{\s})}$ and $\tilde C=\tilde C(C_s,p,\s,q)>0$ is defined in \eqref{eq712}. Observe that $\tilde\varepsilon_1$ depends on $q$ through the sequence $\{q_m\}$. More precisely, $\bar m$ is increasing with respect to $q$, while the quantity $\min_{m=0,...,\bar m}\left(\frac{p( q_m-1)^{1/p}}{p+q_m-2}\right)^pC$ decreases w.r.t. $q_m$. Furthermore, let $\delta_1>0$ be such that
$$
\tilde C\, \delta_1^{\frac{ps(\s-1)}{N(p-2)+ps}}\,+\,\frac{C\,\tilde C}{4}\delta_1^{\frac{ps(\s-1)}{N(p-2)+pq}}<1\,,
$$
where $C>0$ and $\tilde C>0$ are defined in Proposition \ref{prop42} and Lemma \ref{lemma3}, respectively. Then, let $q_0=s$ with $s$ as in Theorem \ref{teo71} and define
\[\varepsilon_1=\varepsilon_1(\s, p, N, C_{s,p}, C_p, q, s)=\min\left\{\tilde \varepsilon_1(\s, p, N, C_{s,p}, C_p, q, s)\,;\delta_1\right\}\,.\]
\end{remark}

\section{$L^q-L^q$ and $L^{q_0}-L^{q}$ estimates}\label{Lq}

Let $x_0,x \in M$. We denote by $r(x)=\textrm{dist}\,(x_0,x)$ the Riemannian distance between $x_0$ and $x$. Moreover, we let $B_R(x_0):=\{x\in M\,: \textrm{dist}\,(x_0,x)<R\}$ be the geodesic ball with centre $x_0 \in M$ and radius $R > 0$. If a reference point $x_0\in M$ is fixed, we shall simply denote by $B_R$ the ball with centre $x_0$ and radius $R$. We also recall that $\mu$ denotes the Riemannian measure on $M$.

\smallskip

For any given function $v$, we define for any $k\in\R^+$
\begin{equation}\label{31}
\begin{aligned}
&T_k(v):=\begin{cases} k\quad &\text{if}\,\,\, v\ge k\,, \\ v \quad &\text{if}\,\,\, |v|< k\,, \\ -k\quad &\text{if}\,\,\, v\le -k\,;\end{cases}.
\end{aligned}
\end{equation}
For every $R>0$, $k>0,$ consider the problem
\begin{equation}\label{eq111}
\begin{cases}
 u_t=\dive \left(|\nabla u|^{p-2}\nabla u \right)+\, T_k(u^{\s}) \quad&\text{in}\,\, B_R\times (0,+\infty) \\
u=0 &\text{in}\,\, \partial B_R\times (0,+\infty)\\
u=u_0 &\text{in}\,\, B_R\times \{0\}, \\
\end{cases}
\end{equation}
where $u_0\in L^\infty(B_R)$, $u_0\geq 0$.
Solutions to problem \eqref{eq111} are meant in the weak sense as follows.

\begin{definition}\label{def31}
Let $p>1$ and $\s>p-1$. Let $u_0\in L^\infty(B_R)$, $u_0\geq 0$. We say that a nonnegative function $u$ is a solution to problem \eqref{eq111} if
$$
u\in L^{\infty}(B_R\times(0,+\infty)), \,\,\, u\in L^2\big((0, T); W^{1,p}_0(B_R)\big) \quad\quad \textrm{ for any }\, T>0,
$$
and for any $T>0$, $\varphi \in C_c^{\infty}(B_R\times[0,T])$ such that $\varphi(x,T)=0$ for every $x\in B_R$, $u$ satisfies the equality:
\begin{equation*}
\begin{aligned}
-\int_0^T\int_{B_R} \,u\,\varphi_t\,d\mu\,dt =&- \int_0^T\int_{B_R}  |\nabla u|^{p-2} \langle\nabla u, \nabla \varphi \rangle \,d\mu\,dt\,+ \int_0^T\int_{B_R} \,T_k(u^\sigma)\,\varphi\,d\mu\,dt \\
& +\int_{B_R} \,u_0(x)\,\varphi(x,0)\,d\mu.
\end{aligned}
\end{equation*}
\end{definition}

\subsection{$L^q-L^q$ estimate for $\s>\s_0$}

First we consider the case $\s>\s_0$ where $\s_0$ has been defined in \eqref{sig0}. Moreover, we assume that the Sobolev inequality \eqref{S} holds on $M$.

\begin{lemma}\label{lemma41}
Assume \eqref{costanti} and, besides, that  $\s>p-1+\frac pN.$ Assume that inequality \eqref{S} holds. Suppose that $u_0\in L^{\infty}(B_R)$, $u_0\ge0$. Let $1< q<\infty$ and assume that
\begin{equation}\label{eq41}
 \|u_0\|_{\textrm L^{\s_0}(B_R)}\,<\,\bar\varepsilon
\end{equation}
with $\bar\varepsilon=\bar\varepsilon(\s,p,q,C_{s,p})>0$ sufficiently small. Let $u$ be the solution of problem \eqref{eq111} in the sense of Definition \ref{def31}, and assume that $u\in C([0,T], L^q(B_R))$ for any $q\in (1,+\infty)$, for any $T>0$. Then
\begin{equation}\label{eq42}
\|u(t)\|_{L^q(B_R)} \le \|u_0\|_{L^q(B_R)}\quad \textrm{ for all }\,\, t>0\,.
\end{equation}
\end{lemma}

Note that the request $u\in C([0,T],L^q(B_R))$ for any $q\in(1,\infty)$, for any $T>0$ is not restrictive, since we will construct solutions belonging to that class. This remark also applies to several other intermediate results below.

\begin{proof}
Since $u_0$ is bounded and $T_k(u^{\s})$ is a bounded and Lipschitz function, by standard results, there exists a unique solution of problem \eqref{eq111} in the sense of Definition \ref{def31}. We now multiply both sides of the differential equation in problem \eqref{eq111} by $u^{q-1}$,
\begin{equation*}\label{eq0}
\int_{B_R}u_t\,u^{q-1}\,dx =\int_{B_R}  \dive \left(|\nabla u|^{p-2}\nabla u \right)\,u^{q-1} \,dx\,+ \int_{B_R}  T_k(u^{\s})\,u^{q-1}\,dx \,.
\end{equation*}
Now, formally integrating by parts in $B_R$. This can be justified by standard tools, by an approximation procedure. We get
\begin{equation}\label{eq43}
\frac{1}{q}\frac{d}{dt}\int_{B_R} u^{q}\,d\mu =-(q-1)\int_{B_R}  u^{q-2}\,|\nabla u|^p \,d\mu\,+ \int_{B_R}  T_k(u^{\s})\,u^{q-1}\,d\mu \,.
\end{equation}
Observe that, thanks to Sobolev inequality \eqref{S}, we have
\begin{equation}\label{eq43b}
\begin{aligned}
\int_{B_R}  u^{q-2}\,|\nabla u|^p \,d\mu&= \left(\frac{p}{p+q-2}\right)^p \int_{B_R}\left |\nabla \left(u^{\frac{p+q-2}{p}}\right)\right|^p \,d\mu\\
&\ge \left(\frac{p}{p+q-2}\right)^pC_{s,p}^p\left( \int_{B_R}  u^{\frac{p+q-2}{p}\frac{pN}{N-p}}\,d\mu \right)^{\frac{N-p}{N}}\,.
\end{aligned}
\end{equation}
Moreover, the last term in the right hand side of \eqref{eq43}, by using the H{\"o}lder inequality with exponents $\frac{N}{N-p}$ and $\frac{N}{p}$, becomes
\begin{equation}\label{eq43c}
\begin{aligned}
\int_{B_R} T_k(u^{\s})\,u^{q-1}\,dx&\le \int_{B_R} u^{\s}\,u^{q-1}\,dx = \int_{B_R}  u^{\s-p+1}\,u^{p+q-2}\,dx \\
&\le \|u(t)\|^{\s-p+1}_{L^{(\s-p+1)\frac{N}{p}}(B_R)} \|u(t)\|^{p+q-2}_{L^{(p+q-2)\frac{N}{N-p}}(B_R)}\,.
\end{aligned}
\end{equation}
Combining \eqref{eq43b} and \eqref{eq43c} we get
\begin{equation}\label{eq44}
\frac{1}{q}\frac{d}{dt} \|u(t)\|^q_{L^q(B_R)}\le -\left[(q-1)\left(\frac{p}{p+q-2}\right)^pC_{s,p}^p-\|u(t)\|^{\s-p+1}_{L^{\s_0}(B_R)}\right] \|u(t)\|^{p+q-2}_{L^{(p+q-2)\frac{N}{N-p}}(B_R)}
\end{equation}
Take $T>0$. Observe that, due to hypotheses \eqref{eq41} and the known continuity in $L^{\s_0}$ of the map $t\mapsto u(t)$ in $[0,T]$, there exists $t_0>0$ such that
$$
\|u(t)\|_{L^{\s_0}(B_R)}\le 2\, \bar\varepsilon\,\,\,\,\,\text{for any}\,\,\,\, t\in [0,t_0]\,.
$$
Hence \eqref{eq44} becomes, for any $t\in (0,t_0]$,
$$
\frac{1}{q}\frac{d}{dt} \|u(t)\|^q_{L^q_{}(B_R)}\le -\left[\left(\frac{p}{p+q-2}\right)^p(q-1)C_{s,p}^p-(2\,\bar\varepsilon)^{\s-p+1} \right] \|u(t)\|^{p+q-2}_{L^{(p+q-2)\frac{N}{N-p}}(B_R)}\,\le 0\,,
$$
where the last inequality is obtained by using \eqref{eq41}. We have proved that $t\mapsto \|u(t)\|_{L^q(B_R)}$ is decreasing in time for any $t\in (0,t_0]$, thus
\begin{equation}\label{eq45}
\|u(t)\|_{L^q(B_R)}\le \|u_0\|_{L^q(B_R)}\quad \text{for any} \,\,\,t\in (0,t_0]\,.
\end{equation}
In particular, inequality \eqref{eq45} follows for the choice $q=\s_0$ in view of hypothesis \eqref{eq41}. Hence we have
$$
\|u(t)\|_{L^{\s_0}(B_R)}\le \|u_0\|_{L^{\s_0}(B_R)}\,<\,\bar\varepsilon \quad \text{for any} \,\,\,\,t\in (0,t_0]\,.
$$
Now, we can repeat the same argument in the time interval $(t_0, t_1]$, where $t_1$ is chosen, due to the continuity of $u$, in such a way that
$$
\|u(t)\|^{\s-p+1}_{L^{\s_0}(B_R)}\le 2\, \bar\varepsilon\,\,\,\,\,\text{for any}\,\,\, t\in (t_0,t_1]\,.
$$
Thus we get
\begin{equation*}
\|u(t)\|_{L^q(B_R)}\le \|u_0\|_{L^q(B_R)}\quad \text{for any} \,\,\,t\in (0,t_1]\,.
\end{equation*}
Iterating this procedure we obtain that $t\mapsto\|u(t)\|_{L^q(B_R)}$ is decreasing in $[0,T]$. Since $T>0$ was arbitrary, the thesis follows.

\end{proof}

\subsection{$L^{q_0}-L^q$ estimate for $\s>\s_0$}

Using a Moser type iteration procedure we prove the following result:

\begin{proposition}\label{prop42}
Assume \eqref{costanti} and, besides, thatn$\sigma>p-1+\frac pN$.  Assume that inequality \eqref{S} holds. Suppose that $u_0\in L^{\infty}(B_R)$, $u_0\ge0$. Let $u$ be the solution of problem \eqref{eq111}, so that $u\in C([0,T], L^q(B_R))$ for any $q\in (1,+\infty)$, for any $T>0$. Let $1< q_0\le q<+\infty$ and assume that
\begin{equation}\label{eq45b}
\|u_0\|_{L^{\s_0}(B_R)}\,\le\,\tilde\varepsilon_0
\end{equation}
for $\tilde\varepsilon_0=\tilde\varepsilon_0(\s,p,N,C_{s,p},q,q_0)$ sufficiently small. Then there exists $C(p,q_0,C_{s,p},\tilde\varepsilon_0,N,q)>0$ such that
\begin{equation}\label{eq46}
\|u(t)\|_{L^q(B_R)} \le C\,t^{-\gamma_q}\|u_0\|^{\delta_q}_{L^{q_0}(B_R)}\quad \textrm{ for all }\,\, t>0\,,
\end{equation}
where
\begin{equation}\label{eq47}
\gamma_q=\left(\frac{1}{q_0}-\frac{1}{q}\right)\frac{N\,q_0}{p\,q_0+N(p-2)}\,,\quad \delta_q=\frac{q_0}{q}\left(\frac{q+\frac{N}{p}(p-2)}{q_0+\frac{N}{p}(p-2)}\right)\,.
\end{equation}
\end{proposition}

\begin{proof}
Let $\{q_n\}$ be the sequence defined in \eqref{eq400}. Let $\bar n$ be the first index such that $q_{\bar n}\ge q$. Observe that $\bar n$ is well defined in view of the mentioned properties of $\{q_n\}$, see \eqref{eq400}. We start by proving a smoothing estimate from $q_0$ to $q_{\bar n}$ using a Moser iteration technique (see also \cite{A}). Afterwords, if $q_{\bar n} \equiv q$ then the proof is complete. Otherwise, if $q_{\bar n} > q$ then, by interpolation, we get the thesis.

Let $t>0$, we define
\begin{equation}\label{eq48}
r=\frac{t}{2^{\overline n}-1} , \quad t_n=(2^n-1)r\,.
\end{equation}
Observe that $t_0=0, \quad t_{\bar n}=t,\quad \{t_n\}$ is an increasing sequence w.r.t. $n$. Now, for any $1\le n\le \overline n$, we multiply equation \eqref{eq111} by $u^{q_{n-1}-1}$ and integrate in $B_R\times[t_{n-1},t_{n}]$. Thus we get
$$
\begin{aligned}
\int_{t_{n-1}}^{t_{n}}\int_{B_R} u_t\,u^{q_{n-1}-1}\,d\mu\,dt &+\int_{t_n}^{t_{n+1}}\int_{B_R}  \dive\left(|\nabla u|^{p-2}\nabla u\right)\,u^{q_{n-1}-1} \,d\mu\,dt\\
&= \int_{t_{n-1}}^{t_{n}}\int_{B_R}  T_k(u^\sigma)\,u^{q_{n-1}-1}\,d\mu\,dt.
\end{aligned}
$$
Then we integrate by parts in $B_R\times[t_{n-1},t_{n}]$. Due to Sobolev inequality \eqref{S} and assumption \eqref{eq45b}, we get
\begin{equation}\label{eq48b}
\begin{aligned}
&\frac{1}{q_{n-1}}\left[\|u(\cdot, t_{n})\|^{q_{n-1}}_{L^{q_{n-1}}(B_R)}-\|u(\cdot, t_{n-1})\|^{q_{n-1}}_{L^{q_{n-1}}(B_R)}\right]\\&\le -\left[\left(\frac{p}{p+q_{n-1}-2}\right)^p(q_{n-1}-1)C_{s,p}^p-2\,\tilde\varepsilon_0 \right] \int_{t_{n-1}}^{t_{n}}\|u(\tau)\|^{p+q_{n-1}-2}_{L^{(p+q_{n-1}-2)\frac{N}{N-p}}(B_R)}\,d\tau,
\end{aligned}
\end{equation}
where we have made use of inequality $T_k(u^{\s})\,\le\,u^{\s}.$ We define $q_n$ as in \eqref{eq400}, so that $(p+q_{n-1}-2)\dfrac{N}{N-p}=q_{n}$. Hence, in view of hypotheses \eqref{eq45b} we can apply Lemma \ref{lemma41} to the integral on the right hand side of \eqref{eq48b}, hence we get
\begin{equation}\label{eq49}
\begin{aligned}
&\frac{1}{q_{n-1}}\left[\|u(\cdot, t_{n})\|^{q_{n-1}}_{L^{q_{n-1}}(B_R)}-\|u(\cdot, t_{n-1})\|^{q_{n-1}}_{L^{q_{n-1}}(B_R)}\right]\\&\le -\left[\left(\frac{p}{p+q_{n-1}-2}\right)^p(q_{n-1}-1)C_{s,p}^p-2\,\tilde\varepsilon_0 \right] \|u(\cdot,t_{n})\|^{p+q_{n-1}-2}_{L^{(p+q_{n-1}-2)\frac{N}{N-p}}(B_R)}|t_{n}-t_{n-1}|.
\end{aligned}
\end{equation}
Observe that
\begin{equation}\label{eq410}
\begin{aligned}
&\left \|u(\cdot, t_{n})\right\|^{q_{n-1}}_{L^{q_{n-1}}(B_R)} \ge 0,\\
& |t_{n}-t_{n-1}|=\frac{2^{n-1}\,t}{2^{\bar n}-1}.
\end{aligned}
\end{equation}
We define
\begin{equation}\label{eq411}
d_{n-1}:=\left[\left(\frac{p}{p+q_{n-1}-2}\right)^p(q_{n-1}-1)C_{s,p}^p-2\tilde\varepsilon_0\right]^{-1}\frac{1}{q_{n-1}}.
\end{equation}
By plugging \eqref{eq410} and \eqref{eq411} into \eqref{eq49} we get
$$
\|u(\cdot, t_{n})\|^{p+q_{n-1}-2}_{L^{(p+q_{n-1}-2)\frac{N}{N-p}}(B_R)}\le \frac{(2^{\bar n}-1)d_n\,}{2^{n-1}\,t}\|u(\cdot,t_{n-1})\|^{q_{n-1}}_{L^{q_{n-1}}(B_R)}.
$$
The latter can be rewritten as
\begin{equation*}\label{eq412}
\|u(\cdot, t_{n})\|_{L^{(p+q_{n-1}-2)\frac{N}{N-p}}(B_R)}\le \left(\frac{(2^{\bar n}-1)d_n}{2^{n-1}}\right)^{\frac{1}{p+q_{n-1}-2}}\,t^{-\frac{1}{p+q_{n-1}-2}}\|u(\cdot,t_{n-1})\|^{\frac{q_{n-1}}{p+q_{n-1}-2}}_{L^{q_{n-1}}(B_R)}.
\end{equation*}
Due to to the definition of the sequence $\{q_n\}$ in \eqref{eq400} we write
\begin{equation}\label{eq413}
\|u(\cdot, t_{n})\|_{L^{q_{n}}(B_R)}\le \left(\frac{(2^{\bar n}-1)d_{n-1}}{2^{n-1}}\right)^{\frac{N}{N-p}\frac{1}{q_{n}}}\,t^{-\frac{N}{N-p}\frac{1}{q_{n}}}\left\|u(\cdot,t_{n-1})\right\|^{\frac{q_{n-1}}{q_{n}}\frac{N}{N-p}}_{L^{q_{n-1}}(B_R)}.
\end{equation}
We define
\begin{equation}\label{eq414}
s:=\frac{N}{N-p}.
\end{equation}
Observe that, for any $1\le n\le \bar n$, we have
\begin{equation}\label{eq415}
\begin{aligned}
\left(\frac{(2^{\bar n}-1)d_{n-1}}{2^{n-1}}\right)^{s}&= \left\{\frac{2^{\bar n}-1}{2^{n-1}}\left[\left(\frac{p}{p+q_{n-1}-2}\right)^p(q_{n-1}-1)C_{s,p}^p-2\,\varepsilon\right]^{-1}\frac{1}{q_{n-1}}\right\}^{s}\\
&= \left[\frac{2^{\bar n}-1}{2^{n-1}}\frac{1}{q_{n-1}(q_{n-1}-1)\left(\dfrac{p}{p+q_{n-1}-2}\right)^pC_{s,p}^p-2\,\varepsilon q_{n-1}}\right]^{s},
\end{aligned}
\end{equation}
and
\begin{equation}\label{eq416}
\frac{2^{\bar n}-1}{2^{n-1}} \le 2^{\bar n+1}\,\,\,\,\,\quad\text{for all}\,\,\, 1\le n\le \bar n.
\end{equation}
Consider the function
$$
g(x):=\left[(x-1)\left(\frac{p}{p+x-2}\right)^pC_{s,p}^p-2\,\varepsilon\right] x\,\,\,\,\,\quad\text{for}\,\,\,q_0\le x \le q_{\bar n},\,\,\,x\in\R.
$$
Observe that, due to \eqref{eq40}, $g(x)>0$ for any $q_0\le x \le q_{\bar n}$. Moreover, $g$ has a minimum in the interval $q_0\le x \le q_{\bar n}$; call $\tilde x$ the point at which the minimum is attained. Then we have
\begin{equation}\label{eq417}
\frac{1}{g(x)}\le \frac{1}{g(\tilde x)} \quad\quad\text{for any }\,\,\,q_0\le x \le q_{\bar n}.
\end{equation}
Thanks to \eqref{eq415}, \eqref{eq416} and \eqref{eq417}, there exist a positive constant $C$, where $C=C(N,C_{s,p},\tilde\varepsilon_0, \bar n,p,q_0)$ such that
\begin{equation}\label{eq418}
\left(\frac{(2^{\bar n}-1)d_{n-1}}{2^{n-1}}\right)^{s} \le C\,,\quad  \text{for all}\,\,\, 1\le n\le\bar n.
\end{equation}
By plugging \eqref{eq414} and \eqref{eq418} into \eqref{eq413} we get, for any $1\le n\le \bar n$
\begin{equation}\label{eq419}
\|u(\cdot, t_{n})\|_{L^{q_n}(B_R)}\le C^{\frac{1}{q_n}}t^{-\frac{s}{q_{n}}}\left\|u(\cdot,t_{n-1})\right\|^{\frac{s\,q_{n-1}}{q_{n}}}_{L^{q_{n-1}}(B_R)}.
\end{equation}
Let us set
$$
U_n:=\|u(\cdot,t_n)\|_{L^{q_n}(B_R)}.
$$
Then \eqref{eq419} becomes
$$
\begin{aligned}
U_n&\le C^{\frac{1}{q_n}}t^{-\frac{s}{q_{n}}}U_{n-1}^{\frac{q_{n-1}s}{q_{n}}}\\
&\le C^{\frac{1}{q_n}}t^{-\frac{s}{q_{n}}}\left [ C^{\frac{s}{q_n}}t^{-\frac{s^2}{q_{n}}} U_{k-2}^{s^2\frac{q_{n-2}}{q_n}}\right] \\
&\le ...\\
&\le C^{\frac{1}{q_n}\sum_{i=0}^{n-1}s^i}t^{-\frac{s}{q_n}\sum_{i=0}^{n-1}s^i} U_0^{s^n\frac{q_0}{q_n}}.
\end{aligned}
$$
We define
\begin{equation}\label{eq420}
\begin{aligned}
&\alpha_n:= \frac{1}{q_n}\sum_{i=0}^{n-1}s^i,\\
&\beta_n:= \frac{s}{q_n}\sum_{i=0}^{n-1}s^i=s\,\alpha_n,\\
&\delta_n:=s^n\frac{q_0}{q_n}.
\end{aligned}
\end{equation}
By substituting $n$ with $\bar n$ into \eqref{eq420} we get
\begin{equation}\label{eq421}
\begin{aligned}
&\alpha_{\bar n}:=\frac{N-p}{p}\frac{A}{ q_{\bar n}},\\
&\beta_{\bar n}:=\frac{N}{p}\frac{A}{q_{\bar n}},\\
&\delta_{\bar n}:=(A+1)\frac{q_0}{q_{\bar n}}.
\end{aligned}
\end{equation}
where $A:=\left(\frac{N}{N-p}\right)^{\bar n}-1$. Hence, in view of \eqref{eq48} and \eqref{eq421}, \eqref{eq419} with $n=\bar n$ yields
\begin{equation}\label{eq422b}
\|u(\cdot, t)\|_{L^{q_{\bar n}}(B_R)}\le C^{\frac{N-p}{p}\frac{A}{q_{\bar n}}}\,t^{-\frac{N}{p}\frac{A}{q_{\bar n}}}\left\|u_0\right\|^{q_{0}\frac{A+1}{q_{\bar n}}}_{L^{q_{0}}(B_R)}.
\end{equation}
We have proved a smoothing estimate from $q_0$ to $q_{\bar n}$.  Observe that if $q_{\bar n}= q$ then the thesis is proved. Now suppose that $q_{\bar n}>q$. Observe that $q_0\le  q < q_{\bar n}$ and define
$$
B:=N(p-2)A+p\,q_0(A+1).
$$
From \eqref{eq422b} and Lemma \ref{lemma41}, we get, by interpolation,
\begin{equation}\label{eq423b}
\begin{aligned}
\|u(\cdot, t)\|_{L^{ q}(B_R)}&\le \|u(\cdot, t)\|_{L^{q_0}(B_R)}^{\theta}\|u(\cdot, t)\|_{L^{q_{\bar n}}(B_R)}^{1-\theta}\\
&\le \|u_0(\cdot)\|_{L^{q_0}(B_R)}^{\theta} C\,t^{-\frac{NA}{B}(1-\theta)}\left\|u_0\right\|^{pq_{0}\frac{A+1}{B}(1-\theta)}_{L^{q_{0}}(B_R)}\\
&=C\,t^{-\frac{N\,A}{B}(1-\theta)}\left\|u_0\right\|^{pq_{0}\frac{A+1}{B}(1-\theta)+\theta}_{L^{q_{0}}(B_R)},
\end{aligned}
\end{equation}
where
\begin{equation}\label{eq424b}
\theta=\frac{q_0}{ q}\left(\frac{q_{\bar n} - q}{q_{\bar n} - q_0}\right).
\end{equation}
Observe that
$$
\begin{aligned}
&\text{(i)}\quad\frac{N A}{B}(1-\theta)=\frac{N}{p}\left(\frac{q-q_0}{q}\right)\frac{1}{q_0+\frac{N}{p}(p-2)};\\
&\text{(ii)}\quad p\,q_0\frac{A+1}{B}(1-\theta)+\theta=\frac{q_0}{q}\frac{q+\frac{N}{p}(p-2)}{q_0+\frac{N}{p}(p-2)}.
\end{aligned}
$$
Combining \eqref{eq423b}, \eqref{eq47} and \eqref{eq424b} we get the claim, noticing that $q$ was arbitrarily in $[q_0, +\infty)$.
\end{proof}

\begin{remark}
One can not let $q\to+\infty$ is the above bound. In fact, one can show that $\varepsilon \longrightarrow 0$ as $q\to\infty$. So in such limit the hypothesis on the norm of the initial datum \eqref{eq40} is satisfied only when $u_0\equiv 0$.
\end{remark}

\subsection{$L^q-L^q$ estimates for $\sigma>p-1$}
We now consider the case $\s>p-1$ and that the Sobolev and Poincar\'e inequalities \eqref{S}, \eqref{P} hold on $M$.

\begin{lemma}\label{lemma71}
Assume \eqref{costanti} and, besides, that $p>2$. Assume that inequalities \eqref{S} and \eqref{P} hold. Suppose that $u_0\in L^{\infty}(B_R)$, $u_0\ge0$. Let $1<q<\infty$ and assume that
\begin{equation}\label{eq713}
\|u_0\|_{L^{\s\frac N{p}}(B_R)}<\bar\varepsilon_1
\end{equation}
for a suitable $\tilde\varepsilon_1=\tilde \varepsilon_1(\s, p, N, C_p, C_{s,p}, q)$ sufficiently small. Let $u$ be the solution of problem \eqref{eq111} in the sense of Definition \ref{def31}, such that in addition $u\in C([0, T); L^q(B_R))$. Then
\begin{equation}\label{eq72}
\|u(t)\|_{L^q(B_R)} \le \|u_0\|_{L^q(B_R)}\quad \textrm{ for all }\,\, t>0\,.
\end{equation}
\end{lemma}

\begin{proof}
Since $u_0$ is bounded and $T_k(u^{\s})$ is a bounded and Lipschitz function, by standard results, there exists a unique solution of problem \eqref{eq111} in the sense of Definition \ref{def31}. We now multiply both sides of the differential equation in problem \eqref{eq111} by $u^{q-1}$, therefore
$$
\int_{B_R} u_t\,u^{q-1}\,d\mu =-\int_{B_R}  \dive(|\nabla u|^{p-2}\,\nabla u)u^{q-1}\,d\mu\,+ \int_{B_R}  T_k(u^{\s})\,u^{q-1}\,d\mu \,.
$$
We integrate by parts. This can again be justified by a standard approximation procedure. By using the fact that $T(u^\sigma)\le u^\sigma$, we can write
\begin{equation}\label{eq73}
\begin{aligned}
\frac{1}{q}\frac{d}{dt}\int_{B_R} u^{q}\,d\mu \le-(q-1)\left(\frac{p}{p+q-2}\right)^p\int_{B_R} \left|\nabla\left( u^{\frac{p+q-2}{p}}\right)\right|^p \,d\mu\,+ \int_{B_R} u^{\s+q-1}\,d\mu.
\end{aligned}
\end{equation}
Now we take $c_1>0$, $c_2>0$ such that $c_1+c_2=1$ so that
\begin{equation}\label{eq79}
\int_{B_R} \left|\nabla\left( u^{\frac{p+q-2}{p}}\right)\right|^p \,d\mu = c_1\, \left\|\nabla\left( u^{\frac{p+q-2}{p}}\right)\right\|_{L^p(B_R)}^p \, + c_2\, \left\|\nabla\left( u^{\frac{p+q-2}{p}}\right)\right\|_{L^p(B_R)}^p.
\end{equation}
Take $\alpha\in (0,1)$. Thanks to \eqref{P}, \eqref{eq79} we get
\begin{equation}\label{eq74}
\begin{aligned}
\int_{B_R} \left|\nabla\left( u^{\frac{p+q-2}{p}}\right)\right|^2 \,d\mu& \ge  c_1\,C_p^p \left\| u\right\|^{p+q-2}_{L^{p+q-2}(B_R)}\, + c_2\, \left\|\nabla\left( u^{\frac{p+q-2}{p}}\right)\right\|_{L^p(B_R)}^p\\
&\ge c_1C_p^p \left\| u\right\|^{p+q-2}_{L^{p+q-2}(B_R)}\, +c_2 \left\|\nabla\left( u^{\frac{p+q-2}{p}}\right)\right\|_{L^p(B_R)}^{p+p\alpha-p\alpha}\\
&\ge c_1C_p^p \left\| u\right\|^{p+q-2}_{L^{p+q-2}(B_R)}\, + c_2C_p^{p\alpha} \left\| u\right\|^{\alpha(p+q-2)}_{L^{p+q-2}(B_R)} \left\|\nabla\left( u^{\frac{p+q-2}{p}}\right)\right\|_{L^p(B_R)}^{p-p\alpha}
\end{aligned}
\end{equation}
Moreover, using the interpolation inequality, H\"{o}lder inequality and \eqref{S}, we have
\begin{equation}\label{eq75}
\begin{aligned}
\int_{B_R}  u^{\s+q-1}\,d\mu,&=\|u\|_{L^{\s+q-1}}^{\s+q-1}\\
&\le \|u\|_{L^{p+q-2}(B_R)}^{\theta(\s+q-1)}\,\|u\|_{L^{\s+p+q-2}(B_R)}^{(1-\theta)(\s+q-1)}\\
&\le \|u\|_{L^{p+q-2}(B_R)}^{\theta(\s+q-1)}\left[\|u\|_{L^{\s\frac{N}{p}}(B_R)}^{\s}\|u\|_{L^{(p+q-2)\frac{N}{N-p}}(B_R)}^{p+q-2}\right]^{\frac{(1-\theta)(\s+q-1)}{\s+p+q-2}}\\
&\le \|u\|_{L^{p+q-2}(B_R)}^{\theta(\s+q-1)}\|u\|_{L^{\s\frac{N}{p}}(B_R)}^{(1-\theta)\frac{\s(\s+q-1)}{\s+p+q-2}} \left(\frac{1}{C_{s,p}}\left\|\nabla \left(u^{\frac{p+q-2}{p}}\right)\right\|_{L^p(B_R)}\right)^{p(1-\theta)\frac{\s+q-1}{\s+p+q-2}}
\end{aligned}
\end{equation}
where $\theta:=\frac{(p-1)(p+q-2)}{\s(\s+q-1)}$. By plugging \eqref{eq74} and \eqref{eq75} into \eqref{eq73} we obtain

\begin{equation}\label{eq710}
\begin{aligned}
\frac{1}{q}\frac{d}{dt}\|u(t)\|_{L^q(B_R)}^{q}
& \le-(q-1)\left(\frac{p}{p+q-2}\right)^p c_1\,C_p^p \left\| u\right\|^{p+q-2}_{L^{p+q-2}(B_R)}\, \\
& - (q-1)\left(\frac{p}{p+q-2}\right)^p c_2\,C_p^{p\alpha} \left\| u\right\|^{\alpha(p+q-2)}_{L^{p+q-2}(B_R)} \left\|\nabla\left( u^{\frac{p+q-2}{p}}\right)\right\|_{L^p(B_R)}^{p-p\alpha} \\
& +\tilde{C}\|u\|_{L^{p+q-2}(B_R)}^{\theta(\s+q-1)}\,\|u\|_{L^{\s\frac{N}{p}}(B_R)}^{(1-\theta)\frac{\s(\s+q-1)}{\s+p+q-2}} \|\nabla \left(u^{\frac{p+q-2}{p}}\right)\|_{L^p(B_R)}^{p(1-\theta)\frac{\s+q-1}{\s+p+q-2}},
\end{aligned}
\end{equation}
where \begin{equation}\label{eq712}
\tilde{C}=\left(\frac{1}{C_{s,p}}\right)^{p(1-\theta)\frac{\s+q-1}{\s+p+q-2}}.
\end{equation}
Let us now fix $\alpha\in (0,1)$ such that
$$
p-p\alpha=p(1-\theta)\frac{\s+q-1}{\s+p+q-2}.
$$
Hence, we have
\begin{equation}\label{eq76}
\alpha=\frac{p-1}{\s}.
\end{equation}
By substituting \eqref{eq76} into \eqref{eq710} we obtain
\begin{equation}\label{eq77}
\begin{aligned}
\frac{1}{q}\frac{d}{dt}\|u(t)\|_{L^q(B_R)}^{q} &\le -(q-1)\left(\frac{p}{p+q-2}\right)^p c_1\,C_p^p \left\| u\right\|^{p+q-2}_{L^{p+q-2}(B_R)}\\
& -  \frac{1}{\tilde C}\left\{ (q-1)\left(\frac{p}{p+q-2}\right)^pC  - \|u\|_{L^{\s\frac{N}{p}}(B_R)}^{\frac{\s(\s+q-1)-(p-1)(p+q-2)}{\s+p+q-2}}\right\}  \\
&\times\left\| u\right\|^{\alpha(p+q-2)}_{L^{p+q-2}(B_R)} \left\|\nabla\left( u^{\frac{p+q-2}{p}}\right)\right\|_{L^p(B_R)}^{p-p\alpha},
\end{aligned}
\end{equation}
where $C$ has been defined in Remark \ref{rem1}. Observe that, due to hypotheses \eqref{eq713} and by the continuity of the solution $u(t)$, there exists $t_0>0$ such that
$$
\left\| u(t)\right\|_{L^{\s\frac N{p}}(B_R)}\le 2\, \tilde\varepsilon_1\,\,\,\,\,\text{for any}\,\,\,\, t\in (0,t_0]\,.
$$
Hence, \eqref{eq77} becomes, for any $t\in (0,t_0]$
\begin{equation*}
\begin{aligned}
\frac{1}{q}\frac{d}{dt}\|u(t)\|_{L^q(B_R)}^{q} &\le  -(q-1)\left(\frac{p}{p+q-2}\right)^p c_1C_p^p \left\| u\right\|^{p+q-2}_{L^{p+q-2}(B_R)}\, \\
& -  \frac{1}{\tilde C}\left\{ (q-1)\left(\frac{p}{p+q-2}\right)^pC  -2\tilde \varepsilon_1^{\frac{\s(\s+q-1)-(p-1)(p+q-2)}{\s+p+q-2}}\right\}  \left\| u\right\|^{\alpha(p+q-2)}_{L^{p+q-2}(B_R)} \left\|\nabla\left( u^{\frac{p+q-2}{p}}\right)\right\|_{L^p(B_R)}^{p-p\alpha}\\
&\\
&\le 0\,,
\end{aligned}
\end{equation*}
 provided $\tilde\varepsilon_1$ is small enough. Hence we have proved that $\|u(t)\|_{L^q(B_R)}$ is decreasing in time for any $t\in (0,t_0]$, thus
\begin{equation}\label{eq78}
\|u(t)\|_{L^q(B_R)}\le \|u_0\|_{L^q(B_R)}\quad \text{for any} \,\,\,t\in (0,t_0]\,.
\end{equation}
In particular, inequality \eqref{eq78} holds $q=\s\frac N{p}$. Hence we have
$$
\|u(t)\|_{L^{\s\frac N{p}}(B_R)}\le \|u_0\|_{L^{\s\frac N{p}}(B_R)}\,<\,\tilde\varepsilon_1\quad \text{for any} \,\,\,\,t\in (0,t_0]\,.
$$
Now, we can repeat the same argument in the time interval $(t_0, t_1]$ where $t_1$ is chosen, by using the continuity of $u(t)$, in such a way that
$$
\left\| u(t)\right\|_{L^{\s\frac N{p}}(B_R)}\le 2\, \tilde\varepsilon_1\,\,\,\,\,\text{for any}\,\,\, t\in (t_0,t_1]\,.
$$
Thus we get
\begin{equation*}
\|u(t)\|_{L^{q}(B_R)}\le \|u_0\|_{L^q(B_R)}\quad \text{for any} \,\,\,t\in (0,t_1]\,.
\end{equation*}
Iterating this procedure we obtain the thesis.

\end{proof}

\subsection{$L^{q_0}-L^q$ estimate for $\s>p-1$}
Using a Moser type iteration procedure we prove the following result:

\begin{proposition}\label{prop71}
Assume \eqref{costanti} and, besides, that $p>2$. Let $M$ be such that \eqref{S} and \eqref{P} hold. Suppose that $u_0\in L^{\infty}(B_R)$, $u_0\ge0$. Let $u$ be the solution of problem \eqref{eq111} in the sense of Definition \ref{def31} such that in addition $u\in C([0,T],L^q(B_R))$ for any $q\in(1,+\infty)$, for any $T>0$. Let $1< q_0\le q<+\infty$ and assume that
\begin{equation}\label{eq711}
\|u_0\|_{L^{\s\frac N{p}}}(B_R)<\tilde{\varepsilon}_1
\end{equation}
for $\tilde{\varepsilon}_1=\tilde{\varepsilon}_1(\s,p,N,C_{s,p},C_p,q,q_0)$ sufficiently small. Then there exists $C(p,q_0,C_{s,p}, \tilde\varepsilon_1, N, q)>0$ such that
\begin{equation}\label{eq715}
\|u(t)\|_{L^q(B_R)} \le C\,t^{-\gamma_q}\|u_0\|^{\delta_q}_{L^{q_0}(B_R)}\quad \textrm{ for all }\,\, t>0\,,
\end{equation}
where
\begin{equation}\label{eq716}
\gamma_q=\frac{q_0}{p-2}\left(\frac{1}{q_0}-\frac{1}{q}\right)\,,\quad \delta_q=\frac{q_0}{q}\,.
\end{equation}
\end{proposition}

\begin{proof}
Arguing as in  the proof of Proposition \ref{prop42}, let $\{q_m\}$ be the sequence defined in \eqref{eq70}. Let $\overline m$ be the first index such that $q_{\overline m}\ge q$. Observe that $\bar m$ is well defined in view of the mentioned properties of $\{q_m\}$, see \eqref{eq70}. We start by proving a smoothing estimate from $q_0$ to $q_{\overline m}$ using again a Moser iteration technique. Afterwords, if $q_{\overline m} \equiv q$ then the proof is complete. Otherwise, if $q_{\overline m} > q$ then, by interpolation, we get the thesis.

\noindent Let $t>0$, we define
\begin{equation}\label{eq717}
r=\frac{t}{2^{\overline m}-1} , \quad t_m=(2^m-1)r\,.
\end{equation}
Observe that
$$t_0=0, \quad t_{\overline m}=t,\quad \{t_m\}\,\text{ is an increasing sequence w.r.t.}\,\,m.$$
\smallskip

\noindent Now, for any $1\le m\le \overline m$, we multiply equation \eqref{eq111} by $u^{q_{m-1}-1}$ and integrate in $B_R\times[t_{m-1},t_{m}]$. Thus we get
$$
\begin{aligned}
\int_{t_{m-1}}^{t_{n}}\int_{B_R}u_t\,u^{q_{m-1}-1}\,d\mu\,d\tau &+\int_{t_{m-1}}^{t_{m}}\int_{B_R}  \dive\left(|\nabla u^{p-2}|\nabla u\right)\,u^{q_{m-1}-1} \,d\mu\,d\tau\\
&\,\,\,= \int_{t_{m-1}}^{t_{m}}\int_{B_R} T_k(u^\sigma)\,u^{q_{m-1}-1}\,d\mu\,d\tau.
\end{aligned}
$$
Then we integrate by parts in $B_R\times[t_{m-1},t_{m}]$, hence we get
\begin{equation*}\label{eq718}
\begin{aligned}
\frac{1}{q_{m-1}}&\left[\|u(\cdot, t_{m})\|^{q_{m-1}}_{L^{q_{m-1}}(B_R)}-\|u(\cdot, t_{m-1})\|^{q_{m-1}}_{L^{q_{m-1}}(B_R)}\right]\\&\le - (q_{m-1}-1)\left(\frac{p}{p+q_{m-1}-2}\right)^p\int_{t_{m-1}}^{t_{m}}  \int_{B_R} \left|\nabla\left( u^{\frac{p+q_{m-1}-2}{p}}\right)\right|^p \,d\mu\,d\tau\\
&\quad\quad+\int_{t_{m-1}}^{t_{m}}\int_{B_R}  u^\sigma\,u^{q_{m-1}-1}\,d\mu\,d\tau.
\end{aligned}
\end{equation*}
where we have made use of inequality
$$
T_k(u^\sigma)\,\le\,u^\sigma.
$$
Now, by arguing as in the proof of Lemma \ref{lemma71}, by using \eqref{eq79} and \eqref{eq74} with $q=q_{m-1}$, we get
\begin{equation*}\label{eq719}
\begin{aligned}
&\int_{B_R}  \left|\nabla\left( u^{\frac{p+q_{m-1}-2}{p}}\right)\right|^pd\mu\\
&\quad\quad\quad \ge c_1C_p^p \left\| u\right\|^{p+q_{m-1}-2}_{L^{p+q_{m-1}-2}(B_R)} + c_2C_p^{p\alpha} \left\| u\right\|^{\alpha(p+q_{m-1}-2)}_{L^{p+q_{m-1}-2}(B_R)} \left\|\nabla\left( u^{\frac{p+q_{m-1}-2}{p}}\right)\right\|_{L^p(B_R)}^{p-p\alpha}
\end{aligned}
\end{equation*}
where $\alpha\in(0,1)$ and $c_1>0$, $c_2>0$ with $c_1+c_2=1$. Similarly, from \eqref{eq75} with $q=q_{m-1}$ we can write
\begin{equation*}\label{eq720}
\begin{aligned}
\int_{B_R}u^\sigma u^{q_{m-1}-1}d\mu&=\|u\|_{L^{p+q_{m-1}-1}(B_R)}^{\s+q_{m-1}-1}\\
&\le \|u\|_{L^{p+q_{m-1}-2}(B_R)}^{\theta(\s+q_{m-1}-1)}\,\|u\|_{L^{\s\frac{N}{p}}(B_R)}^{(1-\theta)\frac{\s(\s+q_{m-1}-1)}{\s+p+q_{m-1}-2}} \\
&\quad \times\left(\frac{1}{C_{s,p}}\left\|\nabla(u^{\frac{p+q_{m-1}-2}{p}})\right\|_{L^p(B_R)}\right)^{p(1-\theta)\frac{\s+q_{m-1}-1}{\s+p+q_{m-1}-2}}
\end{aligned}
\end{equation*}
where $\theta:=\frac{(p-1)(p+q_{m-1}-2)}{\s(\s+q_{m-1}-1)}$.
Now, due to assumption \eqref{eq713}, the continuity of $u$, by choosing $\tilde C$ and $\alpha$ as in \eqref{eq712} and \eqref{eq76} respectively,
we can argue as in the proof of Lemma \ref{lemma71} (see \eqref{eq77}), hence we obtain
\begin{equation}\label{eq723}
\begin{aligned}
\frac{1}{q_{m-1}}&\left[\|u(\cdot, t_{m})\|^{q_{m-1}}_{L^{q_{m-1}}(B_R)}-\|u(\cdot, t_{m-1})\|^{q_{m-1}}_{L^{q_{m-1}}(B_R)}\right]\\
&\le-(q_{m-1}-1)\left(\frac{p}{p+q_{m-1}-2}\right)^pc_1C_p^p \int_{t_{m-1}}^{t_{m}}  \left\| u(\cdot, \tau)\right\|^{p+q_{m-1}-2}_{L^{p+q_{m-1}-2}(B_R)} d\tau \\
& -  \frac{1}{\tilde C}\left\{(q_{m-1}-1)\left(\frac{p}{p+q_{m-1}-2}\right)^pC\,  - 2\tilde{\varepsilon_1}^{\frac{\s(\s+q_{m-1}-1)-(p-1)(p+q_{m-1}-2)}{\s+p+q_{m-1}-2}}\right\}  \\
&\times \int_{t_{m-1}}^{t_m}\left\| u(\cdot,\tau)\right\|^{\alpha(p+q_{m-1}-2)}_{L^{p+q_{m-1}-2}(B_R)} \left\|\nabla\left( u^{\frac{p+q_{m-1}-2}{p}}\right)(\cdot,\tau)\right\|_{L^p(B_R)}^{p-p\alpha}\,d\tau,
\end{aligned}
\end{equation}
where $C$ has been defined in Remark \ref{rem1}. Finally, provided $\tilde\varepsilon_1$ is small enough, \eqref{eq723} can be rewritten as
\begin{equation*}
\begin{aligned}
\frac{1}{q_{m-1}}&\left[\|u(\cdot, t_{m})\|^{q_{m-1}}_{L^{q_{m-1}}(B_R)}-\|u(\cdot, t_{m-1})\|^{q_{m-1}}_{L^{q_{m-1}}(B_R)}\right]\\
&\quad\quad\quad\le-(q_{m-1}-1)\left(\frac{p}{p+q_{m-1}-2}\right)^pc_1C_p^p \int_{t_{m-1}}^{t_{m}}  \left\| u(\cdot, \tau)\right\|^{p+q_{m-1}-2}_{L^{p+q_{m-1}-2}(B_R)} d\tau.
\end{aligned}
\end{equation*}

\noindent We define $q_m$ as in \eqref{eq70}, so that $q_m=p+q_{m-1}-2$. Then, in view of hypothesis \eqref{eq711}, we can apply Lemma \ref{lemma71} to the integral in the right-hand side of the latter, hence we get
\begin{equation}\label{eq724}
\begin{aligned}
\frac{1}{q_{m-1}}&\left[\|u(\cdot, t_{m})\|^{q_{m-1}}_{L^{q_{m-1}}(B_R)}-\|u(\cdot, t_{m-1})\|^{q_{m-1}}_{L^{q_{m-1}}(B_R)}\right]\\
&\quad\quad\quad\le-(q_{m-1}-1)\left(\frac{p}{p+q_{m-1}-2}\right)^pc_1C_p^p\left\| u(\cdot,t_m)\right\|^{q_m}_{L^{q_m}(B_R)} |t_m-t_{m-1}|.
\end{aligned}
\end{equation}
Observe that
\begin{equation}\label{eq725}
\begin{aligned}
&\|u(\cdot, t_{m})\|^{q_{m-1}}_{L^{q_{m-1}}(B_R)}\,\ge\,0,\\
&|t_m-t_{m-1}|=\frac{2^{m-1}t}{2^{\overline m}-1}.
\end{aligned}
\end{equation}
We define
\begin{equation}\label{eq726}
d_{m-1}:=\left(\frac{p}{p+q_{m-1}-2}\right)^{-p}\frac 1{c_1\,C_p^p}\frac{1}{q_{m-1}(q_{m-1}-1)}.
\end{equation}
By plugging \eqref{eq725} and \eqref{eq726} into \eqref{eq724}, we get
$$
\left\| u(\cdot,t_m)\right\|^{q_m}_{L^{q_m}_{\rho}(B_R)}\,\le \,\frac{2^{\bar m}-1}{2^{m-1}t}\,d_{m-1}\|u(\cdot,t_{m-1})\|^{q_{m-1}}_{L^{q_{m-1}}_{\rho}(B_R)}.
$$
The latter can be rewritten as
\begin{equation}\label{eq724b}
\left\| u(\cdot,t_m)\right\|_{L^{q_m}(B_R)}\,\le \,\left(\frac{2^{\bar m}-1}{2^{m-1}}\,d_{m-1}\right)^{\frac 1{q_m}}t^{-\frac1{q_m}}\|u(\cdot,t_{m-1})\|^{\frac{q_{m-1}}{q_m}}_{L^{q_{m-1}}(B_R)}
\end{equation}
Observe that, for any $1\le m\le \bar m$, we have
\begin{equation}\label{eq727}
\begin{aligned}
\frac{2^{\bar m}-1}{2^{m-1}}\,d_{m-1}&=\frac{2^{\bar m}-1}{2^{m-1}}\left(\frac{p}{p+q_{m-1}-2}\right)^{-p}\frac 1{c_1\,C_p^p}\frac{1}{q_{m-1}(q_{m-1}-1)}\\
&\le 2^{\bar m+1}\frac{1}{c_1\,C_p^p}\left(\frac{p+q_{m-1}-2}{p}\right)^{p}\frac{1}{q_{m-1}(q_{m-1}-1)}.
\end{aligned}
\end{equation}
Consider the function
$$
h(x):=\frac {(p+x-2)^p}{x(x-1)}, \quad \text{for}\,\,\,q_0\le x\le q_{\overline m},\quad x\in\R.
$$
Observe that $h(x)\ge0$ for any $q_0\le x\le q_{\overline m}$. Moreover, $h$ has a maximum in the interval $q_0\le x\le q_{\overline m}$, call $\tilde{x}$ the point at which it is attained. Hence
\begin{equation}\label{eq728}
h(x)\le h(\tilde x)\quad \text{for any}\,\,\,q_0\le x\le q_{\overline m},\quad x\in\R.
\end{equation}
Due to \eqref{eq727} and \eqref{eq728}, we can say that there exists a positive constant $C$, where $C=C(C_p,\bar m,p,q_0)$, such that
\begin{equation}\label{eq729}
\frac{2^{\overline m}-1}{2^{m-1}}\,d_{m-1}\le C\quad \text{for all}\,\,1\le m\le \overline m.
\end{equation}
By using \eqref{eq729} and \eqref{eq724b}, we get, for any $1\le m\le \overline m$
\begin{equation}\label{eq730}
\left\| u(\cdot,t_m)\right\|_{L^{q_m}(B_R)}\,\le\, C^{\frac1{q_m}}t^{-\frac1{q_m}}\|u(\cdot,t_{m-1})\|^{\frac{q_{m-1}}{q_m}}_{L^{q_{m-1}}(B_R)}.
\end{equation}
Let us set
$$
U_m:=\left\| u(\cdot,t_m)\right\|_{L^{q_m}(B_R)}
$$
Then \eqref{eq730} becomes
$$
\begin{aligned}
U_m&\le C^{\frac1{q_m}}t^{-\frac1{q_m}}U_{n-1}^{\frac{q_{m-1}}{q_m}}\\
&\le C^{\frac1{q_m}}t^{-\frac1{q_m}}\left[C^{\frac1{q_{m-1}}}t^{-\frac1{q_{m-1}}}U_{m-2}^{\frac{q_{m-2}}{q_{m-1}}}\right]\\
&\le ...\\
&\le C^{\frac m{q_m}}t^{-\frac m{q_m}}U_0^{\frac{q_0}{q_m}}.
\end{aligned}
$$
We define
\begin{equation}\label{eq731}
\alpha_m:=\frac m{q_m},\quad \delta_m:=\frac{q_0}{q_m}.
\end{equation}
Substituting $m$ with $\bar m$ into \eqref{eq731} and in view of \eqref{eq717}, \eqref{eq730} with $m=\overline m$, we have
\begin{equation*}\label{eq732}
\left\| u(\cdot,t)\right\|_{L^{q_{\overline m}}(B_R)}\,\le\, C^{\alpha_{\overline m}}t^{-\alpha_{\overline m}}\left\| u_0\right\|_{L^{q_{0}}(B_R)}^{\delta_{\overline m}}.
\end{equation*}
Observe that if $q_{\overline m}=q$ then the thesis is proved and one has
$$
\alpha_{\overline m}=\frac1{p-2}\left(1-\frac{q_0}{q}\right),\quad \delta_{\overline m}=\frac{q_0}{q}.
$$
Now suppose that $q<q_{\overline m}$, then in particular $q_0\le q\le q_{\overline m}$. By interpolation and Lemma \ref{lemma71} we get
\begin{equation}\label{eq733}
\begin{aligned}
\left\| u(\cdot,t)\right\|_{L^{q}(B_R)}&\le \left\| u(\cdot,t)\right\|_{L^{q_{0}}(B_R)}^{\theta}\left\| u(\cdot,t)\right\|_{L^{q_{\overline m}}(B_R)}^{1-\theta}\\
& \left\| u(\cdot,t)\right\|_{L^{q_{0}}(B_R)}^{\theta}\, C^{\alpha_{\overline m}(1-\theta)}t^{-\alpha_{\overline m}(1-\theta)}\left\| u_0\right\|_{L^{q_{0}}(B_R)}^{\delta_{\overline m}(1-\theta)}\\
&\le C^{\alpha_{\overline m}(1-\theta)}t^{-\alpha_{\overline m}(1-\theta)}\left\| u_0\right\|_{L^{q_{0}}(B_R)}^{\delta_{\overline m}(1-\theta)+\theta},
\end{aligned}
\end{equation}
where
\begin{equation}\label{eq734}
\theta=\frac{q_0}{q}\left(\frac{q_{\overline m}-q}{q_{\overline m}-q_0}\right).
\end{equation}
Combining \eqref{eq716}, \eqref{eq733} and \eqref{eq734}, we get the claim by noticing that $q$ was arbitrary fixed in $[q_0,+\infty)$.

\end{proof}

\section{Auxiliary results}\label{aux}

In what follows, we will deal with solutions $u_R$ to problem \eqref{eq111} for arbitrary fixed $R>0$. For notational convenience, we will simply write $u$ instead of $u_R$ since no confusion will occur in the present section. We define
\begin{equation*}\label{31b}
G_k(v):=v-T_k(v).
\end{equation*}
where $T_k(v)$ has been defined in \eqref{31}. Let $a_1>0$, $a_2>0$ and $t>\tau_1>\tau_2>0$. We consider, for any $i\in\mathbb N\cup\{0\}$, the sequences
\begin{equation}\label{eq32}
\begin{aligned}
&k_i:=a_2+(a_1-a_2)2^{-i}\,;\\
&\theta_i:=\tau_2+(\tau_1-\tau_2)2^{-i}\,;
\end{aligned}
\end{equation}
and the cylinders
\begin{equation}\label{eq32b}
U_i:=B_R\times(\theta_i,t).
\end{equation}
Observe that the sequence $\{\theta_i\}_{i\in \mathbb{N}}$ is monotone decreasing w.r.t. $i$. Furthermore, we define, for any $i\in\mathbb{N}$, the cut-off functions $\xi_i(\tau)$ such that
\begin{equation}\label{eq33}
\xi_i(\tau):=\begin{cases} 1\quad &\theta_{i-1}<\tau<t\\ 0\quad &0<\tau<\theta_i \end{cases}\quad\quad\text{and}\quad\quad|(\xi_i)_{\tau}|\,\le\, \frac{2^i}{\tau_1-\tau_2}\,.
\end{equation}
Finally, we define
\begin{equation}\label{eq34}
S(t):=\sup_{0<\tau<t}\left(\tau\|u(\tau)\|_{L^{\infty}(B_R)}^{\s-1}\right).
\end{equation}

We can now state the following
\begin{lemma}\label{lemma1}
Let $i\in\mathbb{N}$, $k_i$, $\theta_i$, $U_i$ be defined in \eqref{eq32}, \eqref{eq32b} and $R>0$. Let $u$ be a solution to problem \eqref{eq111}. Then, for any $q>1$, we have that
\begin{equation*}\label{eq35}
\sup_{\tau_1<\tau<t}\int_{B_R}[G_{k_0}(u)]^q\,d\mu + \iint_{U_{i-1}}\left|\nabla[G_{k_i}(u)]^{\frac{p+q-2}{p}}\right|^p\,d\mu d\tau\le 2^i\gamma\,C_1\iint_{U_i}[G_{k_{i+1}}(u)]^q\,d\mu d\tau.
\end{equation*}
where $\gamma=\gamma(p,q)$ and
\begin{equation}\label{C1}C_1:=\frac{1}{\tau_1-\tau_2}+\frac{S(t)}{\tau_1}\frac{2a_1}{a_1-a_2}.\end{equation}
\end{lemma}
\begin{proof}
For any $i\in\mathbb{N}$, we multiply both sides of the differential equation in problem \eqref{eq111} by $[G_{k_i}(u)]^{q-1}\xi_i$, $q>1$, and we integrate on the cylinder $U_i$, yielding:
\begin{equation}\label{eq36}
\begin{aligned}
\iint_{U_i} &u_{\tau}\,[G_{k_i}(u)]^{q-1}\xi_i\,d\mu d\tau \\
&=\iint_{U_i}  \dive(|\nabla u|^{p-2}\,\nabla u)[G_{k_i}(u)]^{q-1}\xi_i\,d\mu d\tau\,+ \iint_{U_i}  T_k(u^{\s})\,[G_{k_i}(u)]^{q-1}\xi_i\,d\mu d\tau\,.
\end{aligned}
\end{equation}
We integrate by parts. Thus we write, due to \eqref{eq33},
\begin{equation}\label{eq37}
\begin{aligned}
\iint_{U_i} u_{\tau}\,[G_{k_i}(u)]^{q-1}\xi_i\,d\mu d\tau &= \frac 1q \iint_{U_i} \frac{d}{d\tau}[(G_{k_i}(u))^{q}]\xi_i\,d\mu d\tau\\
&=-\frac 1q \iint_{U_i}[G_{k_i}(u)]^{q}(\xi_i)_{\tau}\,d\mu d\tau+\frac 1q\int_{B_R}[G_{k_i}(u(x,t))]^{q}\,d\mu
\end{aligned}
\end{equation}
Moreover,
\begin{equation}\label{eq38}
\begin{aligned}
-\iint_{U_i}  \dive(|\nabla u|^{p-2}\,\nabla u)[G_{k_i}(u)]^{q-1}&\xi_i\,d\mu d\tau=\iint_{U_i}|\nabla u|^{p-2}\,\nabla u\cdot \nabla[G_{k_i}(u)]^{q-1}\xi_i\,d\mu d\tau\\
&\ge(q-1) \iint_{U_i} [G_{k_i}(u)]^{q-2} |\nabla [G_{k_i}(u)] |^{p}\, \xi_i\,d\mu d\tau.
\end{aligned}
\end{equation}
Now, combining \eqref{eq36}, \eqref{eq37} and \eqref{eq38}, using the fact that $T(u^\sigma)\le u^\sigma$ and \eqref{eq33}, we can write
\begin{equation}\label{eq39}
\begin{aligned}
\frac 1q\int_{B_R}[G_{k_i}(u(x,t))]^{q}\,d\mu &+(q-1) \iint_{U_i} [G_{k_i}(u)]^{q-2} |\nabla [G_{k_i}(u)] |^{p}\, \xi_i\,d\mu d\tau\\
&\le \frac 1q \iint_{U_i}[G_{k_i}(u)]^{q}(\xi_i)_{\tau}\,d\mu d\tau+\iint_{U_i} u^{\s}\,[G_{k_i}(u)]^{q-1}\xi_i\,d\mu d\tau\,\\
&\le \frac{2^i}{\tau_1-\tau_2}\iint_{U_i}[G_{k_i}(u)]^{q}\,d\mu d\tau+\iint_{U_i} u^{\s}\,[G_{k_i}(u)]^{q-1}\xi_i\,d\mu d\tau.
\end{aligned}
\end{equation}
Let us define
$$
\tilde\gamma:=\left[\min\left\{\frac 1q,\,q-1\right\}\right]^{-1},
$$
thus \eqref{eq39} reads
\begin{equation}\label{eq310}
\begin{aligned}
\int_{B_R}[G_{k_i}(u(x,t))]^{q}\,d\mu &+ \iint_{U_{i}} [G_{k_i}(u)]^{q-2} |\nabla [G_{k_i}(u)] |^{p}\xi_i\,d\mu d\tau\\
&\le \tilde\gamma \frac{2^i}{\tau_1-\tau_2}\iint_{U_i}[G_{k_i}(u)]^{q}\,d\mu d\tau+\tilde\gamma \iint_{U_i} u^{\s}\,[G_{k_i}(u)]^{q-1}\xi_id\mu d\tau.
\end{aligned}
\end{equation}
Observe that the sequence $\{k_i\}_{i\in\mathbb{N}}$ is monotone decreasing, hence
$$
G_{k_0}(u)\le G_{k_i}(u)\le G_{k_{i+1}}(u)\le u\quad\quad\text{for all}\,\,\,i\in\mathbb{N}.
$$
Thus \eqref{eq310} can be rewritten as
\begin{equation}\label{eq311}
\begin{aligned}
\int_{B_R}[G_{k_0}(u(x,t))]^{q}\,d\mu &+ \iint_{U_{i-1}} [G_{k_i}(u)]^{q-2} |\nabla [G_{k_i}(u)] |^{p}\,d\mu d\tau\\
&\le \frac{2^i\,\tilde\gamma }{\tau_1-\tau_2}\iint_{U_i}[G_{k_{i+1}}(u)]^{q}\,d\mu d\tau+\tilde\gamma \iint_{U_i} u^{\s}\,[G_{k_{i+1}}(u)]^{q-1}d\mu d\tau.
\end{aligned}
\end{equation}
Let us now define
$$
I:=\tilde\gamma \iint_{U_i} u^{\s-1}\,u\,[G_{k_{i+1}}(u)]^{q-1}d\mu d\tau
$$
Observe that, for any $i\in\mathbb{N}$,
$$
\frac{u}{k_i}\chi_i\,\le\, \frac{u-k_{i+1}}{k_i-k_{i+1}}\chi_i
$$
where $\chi_i$ is the characteristic function of $D_i:=\{(x,t)\in U_i:\,u(x,t)\ge k_i\}$. Then, by using \eqref{eq34}, we get:
\begin{equation}\label{eq312}
\begin{aligned}
I&\le \tilde\gamma \int_{\theta_i}^t \frac{1}{\tau}\tau\|u(\tau)\|_{L^{\infty}(B_R)}^{\s-1}\int_{B_R}u\left[G_{k_{i+1}}(u)\right]^{q-1}\,d\mu d\tau\\
&=\tilde\gamma \int_{\theta_i}^t \frac{1}{\tau}\tau\|u(\tau)\|_{L^{\infty}(B_R)}^{\s-1} \int_{B_R}k_i\frac{u}{k_i}\left[G_{k_{i+1}}(u)\right]^{q-1}\,d\mu d\tau\\
&\le \tilde\gamma \frac{k_i}{k_i-k_{i+1}}S(t)\int_{\theta_i}^{t}\frac{1}{\tau}  \int_{B_R}\left[G_{k_{i+1}}(u)\right]^{q}\,d\mu d\tau.
\end{aligned}
\end{equation}
By substituting \eqref{eq312} into \eqref{eq311} we obtain
\begin{equation*}\label{eq313}
\begin{aligned}
\sup_{\tau_1<\tau<t} &\int_{B_R}[G_{k_0}(u(x,t))]^{q}\,d\mu + \left(\frac p{p+q-2}\right)^p\iint_{U_{i-1}}\left |\nabla [G_{k_i}(u)]^{\frac{p+q-2}{p}} \right|^{p}\,d\mu d\tau\\
&\le \frac{2^i\,\tilde\gamma }{\tau_1-\tau_2}\iint_{U_i}[G_{k_{i+1}}(u)]^{q}\,d\mu d\tau+\frac{k_i\,\tilde\gamma}{k_i-k_{i+1}}\frac{S(t)}{\theta_0} \iint_{U_i} [G_{k_{i+1}}(u)]^{q}d\mu d\tau.
\end{aligned}
\end{equation*}
To proceed further, observe that
$$
\frac{k_i}{k_i-k_{i+1}}=\frac{2^{i+1}a_2}{a_1-a_2}+2, \quad\text{and}\quad \theta_0\equiv\tau_1.
$$
Consequently, by choosing $C_1$ as in \eqref{C1}, we get
\begin{equation*}\label{eq314}
\begin{aligned}
\sup_{\tau_1<\tau<t} &\int_{B_R}[G_{k_0}(u(x,t))]^{q}\,d\mu + \left(\frac p{p+q-2}\right)^p\iint_{U_{i-1}} |\nabla [G_{k_i}(u)]^{\frac{p+q-2}{p}} |^{p}\,d\mu d\tau\\
&\le 2^i\,\tilde\gamma\, C_1\int\int_{U_i}[G_{k_{i+1}}(u)]^{q}\,d\mu d\tau.
\end{aligned}
\end{equation*}
The thesis follows, letting
\begin{equation}\label{gamma}
\gamma:=\left[\min\left\{1;\,\left(\frac{p}{p+q-2}\right)^p\right\}\right]^{-1}\tilde\gamma.
\end{equation}
\end{proof}

\begin{lemma}\label{lemma2}
Assume \eqref{costanti}, let $1<r<q$ and assume that \eqref{S} holds. Let $k_i$, $\theta_i$, $U_i$ be defined in \eqref{eq32}, \eqref{eq32b} and $R>0$. Let $u$ be a solution to problem \eqref{eq111}. Then, for every $i\in\mathbb{N}$ and $\varepsilon>0$, we have
\begin{equation*}\label{eq315}
\begin{aligned}
\sup_{\tau_1<\tau<t}&\int_{B_R}[G_{k_0}(u)]^q\,d\mu + \iint_{U_{i-1}}\left|\nabla[G_{k_i}(u)]^{\frac{p+q-2}{p}}\right|^p\,d\mu d\tau\\
&\le \varepsilon \iint_{U_i} \left|\nabla[G_{k_{i+1}}(u)]^{\frac{p+q-2}{p}}\right|^p \,d\mu d\tau\\
&+ C(\varepsilon)(2^i\gamma C_1)^{\frac{N(p+q-2-r)+pr}{N(p-2)+pr}}(t-\tau_2)\left(\sup_{\tau_2<\tau<t}\int_{B_R}[G_{k_{\infty}}(u)]^r\,d\mu\right)^{\frac{N(p-2)+pq}{N(p-2)+pr}},
\end{aligned}
\end{equation*}
with $C_1$ and $\gamma$ defined as in \eqref{C1} and \eqref{gamma} respectively and for some $C(\varepsilon)>0$.
\end{lemma}

\begin{proof}
Let us fix $q>1$ and $1<r<q$. We define
\begin{equation}\label{eq316}
\alpha:=r\,\frac{N(p-2)+pq}{N(p+q-2-r)+pr}.
\end{equation}
Observe that, since $1<r<q$, one has $0<\alpha<q$. By H\"older inequality with exponents $\frac{pN}{N-p}\left(\frac{p+q-2}{p(q-\alpha)}\right)$ and $\frac{N(p+q-2)}{N(p+\alpha-2)+p(q-\alpha)}$, we thus have:
\begin{equation}\label{eq317}
\begin{aligned}
\int_{B_R}[G_{k_{i+1}}(u)]^q\,d\mu&= \int_{B_R}[G_{k_{i+1}}(u)]^{q-\alpha}\,d\mu+\int_{B_R}[G_{k_{i+1}}(u)]^\alpha\,d\mu\\
&\le\left(\int_{B_R}[G_{k_{i+1}}(u)]^{\left(\frac{p+q-2}{p}\right)\frac{pN}{N-p}}\,d\mu\right)^{\left(\frac{p(q-\alpha)}{p+q-2}\right)\frac{N-p}{pN}}\\
&\quad\times \left(\int_{B_R}[G_{k_{i+1}}(u)]^{\frac{\alpha N(p+q-2)}{N(p+\alpha-2)+p(q-\alpha)}}\,d\mu\right)^{\frac{N(p+\alpha-2)+p(q-\alpha)}{N(p+q-2)}}\\
&\le \left(\left\|[G_{k_{i+1}}(u)]^{\frac{p+q-2}{p}}\right\|_{L^{p^*}(B_R)}\right)^{\frac{p(q-\alpha)}{p+q-2}} \\
&\quad\times \left(\int_{B_R}[G_{k_{i+1}}(u)]^{\frac{\alpha N(p+q-2)}{N(p+\alpha-2)+p(q-\alpha)}}\,d\mu\right)^{\frac{N(p+\alpha-2)+p(q-\alpha)}{N(p+q-2)}}.
\end{aligned}
\end{equation}
By the definition of $\alpha$ in \eqref{eq316} and inequality \eqref{S}, \eqref{eq317} becomes
\begin{equation}\label{eq318}
\begin{aligned}
\int_{B_R}[G_{k_{i+1}}(u)]^q\,d\mu\le\left(\frac{1}{C_{s,p}}\left\|\nabla[G_{k_{i+1}}(u)]^{\frac{p+q-2}{p}}\right\|_{L^{p}(B_R)}\right)^{\frac{p(q-\alpha)}{p+q-2}} \left(\int_{B_R}[G_{k_{i+1}}(u)]^{r}\,d\mu\right)^{\frac\alpha r}.
\end{aligned}
\end{equation}
We multiply both sides of \eqref{eq318} by $2^i \gamma C_1$ with $C_1$ and $\gamma$ as in \eqref{C1} and \eqref{gamma}, respectively. Then, we apply Young's inequality with exponents $\frac{p+q-2}{q-\alpha}$ and $\frac{p+q-2}{p+\alpha-2}$ to get:
+
\begin{equation}\label{eq319}
\begin{aligned}
&2^i\,\gamma C_1\int_{B_R}[G_{k_{i+1}}(u)]^q\,d\mu\\
&\le \varepsilon\int_{B_R}\left|\nabla[G_{k_{i+1}}(u)]^{\frac{p+q-2}{p}}\right|^p\,d\mu+C(\varepsilon)(2^i\gamma C_1)^{\frac{p+q-2}{p+\alpha-2}}\left(\int_{B_R}[G_{k_{i+1}}(u)]^{r}\,d\mu\right)^{\frac\alpha r\frac{p+q-2}{p+\alpha-2}}
\end{aligned}
\end{equation}
Define
$$
\lambda:=\frac\alpha r\left(\frac{p+q-2}{p+\alpha-2}\right)=\frac{N(p-2)+pq}{N(p-2)+pr}.
$$
Observe that $\lambda>1$ since $r<q$. By Lemma \ref{lemma1},
\begin{equation}\label{eq320}
\begin{aligned}
 \sup_{\tau_1<\tau<t}\int_{B_R}[G_{k_0}(u)]^q\,d\mu + \iint_{U_{i-1}}\left|\nabla[G_{k_i}(u)]^{\frac{p+q-2}{p}}\right|^p\,d\mu d\tau\le 2^i\gamma C_1\int_{\theta_i}^{t}&\int_{B_R}[G_{k_{i+1}}(u)]^q\,d\mu d\tau
\end{aligned}
\end{equation}
Moreover, let us integrate inequality \eqref{eq319} in the time interval $\tau\in(\theta_i,t)$. Then, we observe that
\begin{equation}\label{eq321}
\begin{aligned}
 C(\varepsilon)(2^i\gamma C_1)^{\frac{p+q-2}{p+\alpha-2}}& \int_{\theta_i}^{t}\left(\int_{B_R}[G_{k_{i+1}}(u)]^{r}\,d\mu\right)^{\lambda}\,d\tau\\
 &\le C(\varepsilon)(2^i\gamma C_1)^{\frac{p+q-2}{p+\alpha-2}}(t-\tau_2)\left(\sup_{\tau_2<\tau<t}\int_{B_R}[G_{k_{i+1}}(u)]^r\,d\mu\right)^\lambda
 \end{aligned}
\end{equation}
where we have used that $\tau_2<\theta_i$ for every $i\in\mathbb{N}$.
Finally, we substitute \eqref{eq320} and \eqref{eq321} into \eqref{eq319}, thus we get
$$
\begin{aligned}
\sup_{\tau_1<\tau<t}&\int_{B_R}[G_{k_0}(u)]^q\,d\mu + \iint_{U_{i-1}}\left|\nabla[G_{k_i}(u)]^{\frac{p+q-2}{p}}\right|^p\,d\mu d\tau\\
&\le \varepsilon\iint_{U_i}\left|\nabla[G_{k_{i+1}}(u)]^{\frac{p+q-2}{p}}\right|^p\,d\mu d\tau+C(\varepsilon)(2^i\gamma C_1)^{\frac{p+q-2}{p+\alpha-2}}(t-\tau_2)\left(\sup_{\tau_2<\tau<t}\int_{B_R}[G_{k_{i+1}}(u)]^r\,d\mu\right)^\lambda
\end{aligned}
$$
The thesis follows by noticing that, for any $i\in\mathbb{N}$
$$
G_{k_i}(u)\le G_{k_{i+1}}(u)\le \ldots \le G_{k_\infty}(u),
$$
and that
$$
\frac{p+q-2}{p+\alpha-2}=\frac{N(p+q-2-r)+pr}{N(p-2)+pr}.
$$
\end{proof}

\begin{proposition}\label{lemma3}
Assume that \eqref{costanti} and \eqref{S} holds. Let $S(t)$ be defined as in \eqref{eq34}. Let $u$ be a solution to problem \eqref{eq111}. Suppose that, for all $t\in(0,T)$, $$S(t)\le 1.$$ Let $r\ge1$, then there exists $k=k(p,r)$ such that
\begin{equation*}\label{eq322}
\|u(x,\tau)\|_{L^{\infty}\left(B_R\times\left(\frac t2,t\right)\right)}\,\le\, k \, t^{-\frac{N}{N(p-2)+pr}}\left[\sup_{\frac t4<\tau<t}\int_{B_R}u^r\,d\mu\right]^{\frac{p}{N(p-2)+pr}},
\end{equation*}
for all $t\in(0,T)$.
\end{proposition}

\begin{proof}
Let us define, for any $j\in\mathbb{N}$,
\begin{equation}\label{eq323}
J_i:=\iint_{U_{i}}\left|\nabla\left[G_{k_{i+1}}(u)\right]^{\frac{p+q-2}{p}}\right|^p\,d\mu\,dt,
\end{equation}
where $G_k$, $\{k_i\}_{i\in\mathbb{N}}$ and $U_i$ have been defined in \eqref{31}, \eqref{eq32} and \eqref{eq32b} respectively. Let us fix $1\le r<q $ and define
\begin{equation*}\label{eq324}
\beta:=\frac{N(p+q-2-r)+pr}{N(p-2)+pr}.
\end{equation*}
By means of Lemma \ref{lemma2} and \eqref{eq323}, we can write, for any $i\in\mathbb{N}\cup\{0\}$
\begin{equation}\label{eq325}
\begin{aligned}
\sup_{\tau_1<\tau<t}&\int_{B_R}[G_{k_0}(u)]^q\,d\mu + J_0\\
&\le \varepsilon J_1+ C(\varepsilon)(2\gamma C_1)^{\beta}(t-\tau_2)\left(\sup_{\tau_2<\tau<t}\int_{B_R}[G_{k_{\infty}}(u)]^r\,d\mu\right)^{\frac{N(p-2)+pq}{N(p-2)+pr}}\\
&\le \varepsilon \left\{J_2+ C(\varepsilon)(2^2\gamma C_1)^{\beta}(t-\tau_2)\left(\sup_{\tau_2<\tau<t}\int_{B_R}[G_{k_{\infty}}(u)]^r\,d\mu\right)^{\frac{N(p-2)+pq}{N(p-2)+pr}}\right\} \\
&\quad\quad+ C(\varepsilon)(2 \gamma C_1)^{\beta}(t-\tau_2)\left(\sup_{\tau_2<\tau<t}\int_{B_R}[G_{k_{\infty}}(u)]^r\,d\mu\right)^{\frac{N(p-2)+pq}{N(p-2)+pr}}\\
&\le \ldots\\
&\le \varepsilon^{i} J_{i}+\sum_{j=0}^{i-1}(2^{\beta}\varepsilon)^{j}(2 \gamma C_1)^\beta\,C(\varepsilon)(t-\tau_2)\left(\sup_{\tau_2<\tau<t}\int_{B_R}[G_{k_{\infty}}(u)]^r\,d\mu\right)^{\frac{N(p-2)+pq}{N(p-2)+pr}}.
\end{aligned}
\end{equation}
Fix now $\varepsilon>0$ such that $\varepsilon 2^\beta<\frac 12$. Taking the limit as $i\longrightarrow+\infty$ in \eqref{eq325} we have:
\begin{equation}\label{eq326}
\sup_{\tau_1<\tau<t}\int_{B_R}[G_{k_0}(u)]^q\,d\mu\,\le\,\tilde C(2 \gamma C_1)^{\beta}(t-\tau_2)\left(\sup_{\tau_2<\tau<t}\int_{B_R}[G_{k_{\infty}}(u)]^r\,d\mu\right)^{\frac{N(p-2)+pq}{N(p-2)+pr}}.
\end{equation}
Observe that, due to the definition of the sequence $\{k_i\}_{i\in\mathbb{N}}$ in \eqref{eq32}, one has
$$
\begin{aligned}
&k_0=a_1\,,\quad\quad \quad \quad\quad \quad  k_{\infty}=a_2\,;\\
&G_{k_0}(u)=G_{a_1}(u)\,, \quad\quad G_{k_\infty}(u)=G_{a_2}(u)\,.
\end{aligned}
$$
For $n\in\mathbb{N}\cup\{0\}$, consider, for some $C_0>0$ to be fixed later, the following sequences
\begin{equation}\label{eq327}
\begin{aligned}
&t_n=\frac 12 t(1-2^{-n-1})\,;\\
&h_n=C_0(1-2^{-n-1})\,;\\
&\overline{h}_n=\frac 12(h_n+h_{n+1})\,.
\end{aligned}
\end{equation}
Let us now set in  \eqref{eq326}:
\begin{equation}\label{eq328}
\tau_1=t_{n+1}\,;\quad \tau_2=t_n\,;\quad a_1=\overline h_n\,;\quad a_2=h_n\,.
\end{equation}
Then the coefficient $C_1$ defined in \eqref{C1}, by \eqref{eq327} and \eqref{eq328}, satisfies, since for any $t\in(0,T)$ one has $S(t)\le1$,
$$
2C_1\le \frac{C_2^n}{t}\quad \text{for some}\,\,C_2>1.
$$
Due to the latter bound and to \eqref{eq328}, \eqref{eq326} reads
\begin{equation}\label{eq329}
\sup_{t_{n+1}<\tau<t}\int_{B_R}[G_{\overline h_n}(u)]^q\,d\mu\,\le\,\tilde C\,\gamma \,C_2^{n\beta}t^{-\beta+1}\left(\sup_{t_n<\tau<t}\int_{B_R}[G_{h_{n}}(u)]^r\,d\mu\right)^{\frac{N(p-2)+pq}{N(p-2)+pr}}.
\end{equation}
Furthermore, observe that
\begin{equation}\label{eq330}
\int_{B_R}[G_{h_{n+1}}(u)]^r\,d\mu\,\le(h_{n+1}-\overline h_n)^{r-q}\int_{B_R}\left[G_{\overline h_n}(u)\right]^q\,d\mu.
\end{equation}
By combining together \eqref{eq329} and \eqref{eq330}, we derive the following inequalities:
\begin{equation}\label{eq331}
\begin{aligned}
\sup_{t_{n+1}<\tau<t}\int_{B_R}&[G_{h_{n+1}}(u)]^r\,d\mu\,\le (h_{n+1}-\overline h_n)^{r-q}\sup_{t_{n+1}<\tau<t}\int_{B_R}\left[G_{\overline h_n}(u)\right]^q\,d\mu\\
&\le \tilde C\,\gamma\,C_2^{n\beta}\left(\frac{h_{n+1}-h_n}{2}\right)^{r-q}t^{-\beta+1}\left(\sup_{t_n<\tau<t}\int_{B_R}[G_{h_{n}}(u)]^r\,d\mu\right)^{\frac{N(p-2)+pq}{N(p-2)+pr}}.
\end{aligned}
\end{equation}
Let us finally define
$$
Y_n:=\sup_{t_n<\tau<t}\int_{B_R}[G_{h_{n}}(u)]^r\,d\mu.
$$
Hence, by using \eqref{eq327}, \eqref{eq331} reads,
\begin{equation}\label{eq332}
\begin{aligned}
Y_{n+1}&\le \tilde C\,\gamma\, C_2^{n\beta}\left(\frac{h_{n+1}-h_n}{2}\right)^{r-q}\,t^{-\beta+1}\,Y_n^{\frac{N(p-2)+pq}{N(p-2)+pr}}\\
&\le \tilde C\,\gamma\,C_2^{n\beta}\,2^{(n+3)(q-r)}\,C_0^{r-q}\,t^{-\beta+1}\,Y_n^{\frac{N(p-2)+pq}{N(p-2)+pr}}\,\\
&\le k^{n(q-r)}\,C_0^{r-q}\,t^{-\beta+1}\,Y_n^{\frac{N(p-2)+pq}{N(p-2)+pr}}\,,
\end{aligned}
\end{equation}
for some $k=k(p,r)>1$. From \cite[Chapter 2, Lemma 5.6]{LSU} it follows that
\begin{equation}\label{eq333}
Y_n\longrightarrow 0\quad\text{as}\,\,\,n\to+\infty,
\end{equation}
provided
\begin{equation}\label{eq334}
C_0^{r-q}\,t^{-\beta+1}\,Y_0^{\frac{N(p-2)+pq}{N(p-2)+pr}-1}\le k^{r-q}.
\end{equation}
Now, \eqref{eq333}, in turn, reads
$$
\|u\|_{L^{\infty}\left(B_R\times\left(\frac t2,t\right)\right)}\,\le\, C_0.
$$
Moreover, \eqref{eq334} is fulfilled if
$$
C_0=k\,t^{\frac{-\beta+1}{q-r}}\,Y_0^{\left(\frac{N(p-2)+pq}{N(p-2)+pr}-1\right)\left(\frac1{q-r}\right)}\le k\,t^{-\frac{N}{N(p-2)+pr}}\left[\sup_{\frac t4<\tau<t}\int_{B_R}u^r\,d\mu\right]^{\frac{p}{N(p-2)+pr}}.
$$
This concludes the proof.
\end{proof}

\section{Proof of Theorem \ref{teo1}}\label{dim1}
By Lemma \ref{lemma3}, using the same arguments as in the proof of \cite[Lemmata 4 and 5, and subsequent remarks]{MT1}, we get the following result.

\begin{lemma}\label{lemma4}
Assume \eqref{costanti} and $\s>p-1+\frac pN$. Suppose that \eqref{S} and \eqref{epsilon0} hold. Let $S(t)$ be defined as in \eqref{eq34}. Define
\begin{equation}\label{eq335}
T:=\sup\left\{t>0:\,S(t)\le\,1\right\}.
\end{equation}
Then
\begin{equation*}\label{eq336}
T=+\infty.
\end{equation*}
\end{lemma}

\begin{proof}[Proof of Theorem \ref{teo1}]
Let $\{u_{0,h}\}_{h\ge 0}$ be a sequence of functions such that
\begin{equation*}
\begin{aligned}
&(a)\,\,u_{0,h}\in L^{\infty}(M)\cap C_c^{\infty}(M) \,\,\,\text{for all} \,\,h\ge 0, \\
&(b)\,\,u_{0,h}\ge 0 \,\,\,\text{for all} \,\,h\ge 0, \\
&(c)\,\,u_{0, h_1}\leq u_{0, h_2}\,\,\,\text{for any } h_1<h_2,  \\
&(d)\,\,u_{0,h}\longrightarrow u_0 \,\,\, \text{in}\,\, L^{s}(M)\cap L^1(M)\quad \textrm{ as }\, h\to +\infty\,,\\
\end{aligned}
\end{equation*}

Observe that, due to assumptions $(c)$ and $(d)$, $u_{0,h}$ satisfies \eqref{epsilon0}. For any $R>0$, $k>0$, $h>0$, consider the problem
\begin{equation}\label{5}
\begin{cases}
u_t= \dive\left(|\nabla u|^{p-2}\nabla u\right) +T_k(u^{\s}) &\text{in}\,\, B_R\times (0,+\infty)\\
u=0& \text{in}\,\, \partial B_R\times (0,\infty)\\
u=u_{0,h} &\text{in}\,\, B_R\times \{0\}\,. \\
\end{cases}
\end{equation}
From standard results it follows that problem \eqref{5} has a solution $u_{h,k}^R$ in the sense of Definition \ref{def31}. In addition, $u^R_{h,k}\in C\big([0, T]; L^q(B_R)\big)$ for any $q>1$.

\noindent (i) In view of Proposition \ref{lemma3} and Lemma \ref{lemma4}, the solution $u_{h,k}^R$ to problem \eqref{5} satisfies estimate \eqref{eq332} for any $t\in(0,+\infty)$, uniformly w.r.t. $R$, $k$ and $h$. By standard arguments we can pass to the limit as $R\to\infty$, $k\to\infty$ and $h\to\infty$ and we obtain a solution $u$ to equation \eqref{problema} satisfying \eqref{eq21tot}.
\medskip

\noindent (ii) Due to Proposition \ref{prop42}, $u_{h,k}^R$ to problem \eqref{5} satisfies estimate \eqref{eq46} for any $t\in(0,+\infty)$, uniformly w.r.t. $R$, $k$ and $h$. Thus, the solution $u$ fulfills \eqref{eq23}.
\medskip

\noindent (iii) We now furthermore suppose that $u_{0,h}\in L^q(M)$ and $u_{0,h}\longrightarrow u_0$ in $L^{q}(M)$. Due to Proposition \ref{lemma41}, the solution $u_{h,k}^R$ to problem \eqref{5} satisfies estimate \eqref{eq42} for any $t\in(0,+\infty)$, uniformly w.r.t. $R$, $k$ and $h$. Thus, the solution $u$ also fulfills \eqref{eq22}.
\smallskip

This completes the proof.

\end{proof}

\section{Proof of Theorem \ref{teo11}}\label{dim2}
To prove Theorem \ref{teo11} we need the following two results.
\begin{lemma}\label{lemma4old}
Assume \eqref{costanti} and, moreover, that $\s>p-1+\frac pN$. Assume that inequality \eqref{S} holds. Let $u$ be a solution of problem \eqref{eq111} with $u_0\in  L^{\infty}(B_R)$, $u_0\ge0$, such that
$$
\|u_0\|_{L^{\s_0}(B_R)}\le\varepsilon_2,
$$
for $\varepsilon_2=\varepsilon_2(\s,p,N,C_{s,p},\s_0)>0$ sufficiently small and $\s_0$ as in \eqref{sig0}. Let $S(t)$ and $T$ be defined as in \eqref{eq34} and \eqref{eq335} respectively. Then
\begin{equation*}\label{eq335old}
T=+\infty.
\end{equation*}
\end{lemma}
\begin{proof}
We suppose by contradiction that $T<+\infty$. Then, by \eqref{eq335} and \eqref{eq34}, we can write:
\begin{equation}\label{eq336old}
1=S(T)=\sup_{0<t<T}\,t\|u(t)\|_{L^{\infty}(B_R)}^{\s-1}.
\end{equation}
Due to Lemma \eqref{lemma3} with the choice $r=q>\s_0$, \eqref{eq336old} reduces to
\begin{equation}\label{eq337old}
\begin{aligned}
1=S(T)&\le \,\sup_{0<t<T}\,t\left\{k\,t^{-\frac{N}{N(p-2)+pq}}\left(\sup_{\frac t4<\tau<t}\int_{B_R}u^q\,d\mu\right)^{\frac{p}{N(p-2)+pq}}\right\}^{(\s-1)}\\
&\le  \,\sup_{0<t<T}\,k\,t^{1-\frac{N(\s-1)}{N(p-2)+pq}}\left(\sup_{\frac t4<\tau<t}\left\|u(\tau)\right\|_{L^{q}(B_R)}^{\frac{q\,p(\s-1)}{N(p-2)+pq}}\right)\,.
\end{aligned}
\end{equation}
Define
\begin{equation}\label{eq338old}
I_1:= \sup_{\frac t4<\tau<t}\left\|u(\tau)\right\|_{L^{q}(B_R)}^{\frac{pq(\s-1)}{N(p-2)+pq}}.
\end{equation}
In view of the choice $q>\s_0$, we can apply Proposition \ref{prop42} with $q_0=\s_0$ to \eqref{eq338old}, thus we get
\begin{equation}\label{eq339old}
\begin{aligned}
I_1&\le \sup_{\frac t4<\tau<t}\left[C\,t^{-\gamma_q}\,\|u_0\|^{\delta_q}_{L^{q_0}(B_R)}\right]^{\frac{pq(\s-1)}{N(p-2)+pq}}\\
& \le \,C\,t^{-\gamma_q\frac{pq(\s-1)}{N(p-2)+pq}}\,\|u_0\|^{\delta_q\frac{pq(\s-1)}{N(p-2)+pq}}_{L^{q_0}(B_R)}\,,
\end{aligned}
\end{equation}
where $\gamma_q$ and $\delta_q$ are defined in \eqref{eq47}. By substituting \eqref{eq339old} into \eqref{eq337old} we get
\begin{equation*}\label{eq340old}
1=S(T)\le \,C\,k\sup_{0<t<T}\, t^{1-\frac{N(\s-1)}{N(p-2)+pq}-\gamma_q\frac{pq(\s-1)}{N(p-2)+pq}}\|u_0\|^{\delta_q\frac{pq(\s-1)}{N(p-2)+pq}}_{L^{q_0}(B_R)} \,.
\end{equation*}
Observe that
$$
\begin{aligned}
&1-\frac{N(\s-1)}{N(p-2)+pq}-\gamma_q\frac{pq(\s-1)}{N(p-2)+pq}=0;\\
& \delta_q\frac{pq(\s-1)}{N(p-2)+pq}=\s-p+1>0\,;
\end{aligned}
$$
hence
$$
1=S(T)< C\,\tilde C\, \varepsilon_2^{\s-p+1}\,.
$$
Provided $\varepsilon_2$ is sufficiently small, a contradiction, i.e. $1=S(T)<1$. Thus $T=+\infty$.
\end{proof}

\begin{proposition}\label{prop3old}
ssume \eqref{costanti} and, moreover, that $\s>p-1+\frac pN$. Let $u$ be the solution to problem \eqref{eq111} with $u_0\in{\textrm L}^{\infty}(B_R)$, $u_0\ge0$. Let $\s_0$ be defined in \eqref{sig0} and $q>\s_0.$ Assume that
$$
\|u_0\|_{\textrm L^{\s_0}(B_R)}\,<\,\varepsilon_2
$$
with $\varepsilon_2=\varepsilon_2(\s,p,N, C_{s,p},\s_0)>0$ sufficiently small. Then, for some $C=C(N,\s,p,q,\s_0)>0$:
\begin{equation}\label{eq21old}
\|u(t)\|_{L^{\infty}(B_R)}\le C\, t^{-\frac{1}{\s-1}}\,\|u_0\|_{L^{\s_0}(B_R)}^{1-\frac{p-2}{\s-1}}\,\quad\text{for any $t\in(0,+\infty)$.}
\end{equation}
\end{proposition}

\begin{proof}
Due to Lemma \ref{lemma4old}, $$S(t)\le 1\quad \text{for all}\,\,\,t\in(0,+\infty).$$ Therefore, by Lemma \ref{lemma3} and Proposition \ref{prop42} with $q_0=\s_0$, for all $t\in (0,+\infty)$
\begin{equation}\label{eq341old}
\begin{aligned}
\|u(t)\|_{L^{\infty}(B_R)}&\le \|u\|_{L^{\infty}\left(B_R\times\left(\frac t2,t\right)\right)}\,\\
&\le\,k\,t^{-\frac{N}{N(p-2)+pq}}\left[\sup_{\frac t4<\tau<t}\|u(\tau)\|_{L^q(B_R)}^q\right]^{\frac{p}{N(p-2)+pq}}\\
&\le\,C\,t^{-\frac{N}{N(p-2)+pq}-\gamma_q\frac{pq}{N(p-2)+pq}}\|u_0\|_{L^{\s_0}(B_R)}^{\delta_q\frac{pq}{N(p-2)+pq}}\,,
\end{aligned}
\end{equation}
where $C=C(\s,p,N,q,\s_0)>0$, $\gamma_q$ and $\delta_q$ as in \eqref{eq47} with $q_0=\s_0$.
Observe that
\begin{equation}\label{eq342old}
-\frac{N}{N(p-2)+pq}-\gamma_q\frac{pq}{N(p-2)+pq}=-\frac 1{\s-1}\,,
\end{equation}
and
\begin{equation}\label{eq343old}
\delta_q\frac{pq}{N(p-2)+pq}=\frac{\s-p+1}{\s-1}\,.
\end{equation}
By combining \eqref{eq341old} with \eqref{eq342old} and \eqref{eq343old} we get the thesis.
\end{proof}

\begin{proof}[Proof of Theorem \ref{teo11}] We use the same argument discussed in the proof of Theorem \ref{teo1}. In fact, let $\{u_{0,l}\}_{l\ge 0}$ be a sequence of functions such that
\begin{equation*}
\begin{aligned}
&(a)\,\,u_{0,l}\in L^{\infty}(M)\cap C_c^{\infty}(M) \,\,\,\text{for all} \,\,l\ge 0, \\
&(b)\,\,u_{0,l}\ge 0 \,\,\,\text{for all} \,\,l\ge 0, \\
&(c)\,\,u_{0, l_1}\leq u_{0, l_2}\,\,\,\text{for any } l_1<l_2,  \\
&(d)\,\,u_{0,l}\longrightarrow u_0 \,\,\, \text{in}\,\, L^{\s_0}(M)\quad \textrm{ as }\, l\to +\infty\,,\\
\end{aligned}
\end{equation*}
where $\s_0$ has been defined in \eqref{sig0}. Observe that, due to assumptions $(c)$ and $(d)$, $u_{0,l}$ satisfies \eqref{epsilon000}. For any $R>0$, $k>0$, $l>0$, we consider problem \eqref{5} with the sequence $u_{0,h}$ replaced by the sequence $u_{0,l}$. From standard results it follows that problem \eqref{5} has a solution $u_{l,k}^R$ in the sense of Definition \ref{def31}; moreover, $u^R_{l,k}\in C\big([0, T]; L^q(B_R)\big)$ for any $q>1$.

Due to Proposition \ref{prop3old}, Proposition \ref{prop42} and Lemma \ref{lemma41}, the solution $u_{l,k}^R$ to problem \eqref{5} satisfies estimates \eqref{eq21old}, \eqref{eq46} and \eqref{eq42} for $t\in(0,+\infty)$, uniformly w.r.t. $R$, $k$ and $l$. Thus, by standard arguments we can pass to the limit as $R\to\infty$, $k\to\infty$ and $l\to\infty$ and we obtain a solution $u$ to equation \eqref{problema} satisfying \eqref{eq21totold}, \eqref{eq23} and \eqref{eq22}.
\end{proof}

\section{Proof of Theorem \ref{teo71}}\label{dim3}
\begin{lemma}\label{lemma5}
Assume \eqref{costanti}, $p>2$, and $q>\max\left\{\s_0,1\right\}$. Let $u$ be a solution to problem \eqref{eq111} with $u_0\in  L^{\infty}(B_R)$, $u_0\ge0$, such that
\begin{equation}\label{eq60}
\|u_0\|_{L^{q}(B_R)}\le\delta_1,
\end{equation}
for $\delta_1>0$ sufficiently small. Let $S(t)$ be as in \eqref{eq34}, then
\begin{equation}\label{eq61}
T:=\sup\{t>0:\,S(t)\le\,1\}>1.
\end{equation}
\end{lemma}

\begin{proof}
By \eqref{eq34} and \eqref{eq61} one has
\begin{equation}\label{eq63}
1=S(T)=\sup_{0<t<T}\,t\|u(t)\|_{L^{\infty}(B_R)}^{\s-1}.
\end{equation}
By Lemma \eqref{lemma3} applied with $r=q>\max\left\{\frac Np(\s-p+1),1\right\}$, \eqref{eq63} gives
\begin{equation}\label{eq64}
\begin{aligned}
1=S(T)&\le \sup_{0<t<T}\,t\left\{k\,t^{-\frac{N}{N(p-2)+pq}}\left(\sup_{\frac t4<\tau<t}\int_{B_R}u^q\,d\mu\right)^{\frac{p}{N(p-2)+pq}}\right\}^{(\s-1)}\\
&\le  \sup_{0<t<T}\, k\,t^{1-\frac{N(\s-1)}{N(p-2)+pq}}\left(\sup_{\frac t4<\tau<t}\left\|u(\tau)\right\|_{L^{q}(B_R)}^{\frac{q\,p(\s-1)}{N(p-2)+pq}}\right)\,.
\end{aligned}
\end{equation}
By applying Proposition \ref{prop71} to \eqref{eq64} and due to \eqref{eq60}, we get
\begin{equation*}\label{eq65}
\begin{aligned}
1=S(T)&\le \sup_{0<t<T}\, k\,t^{1-\frac{N(\s-1)}{N(p-2)+pq}}\left\|u_0\right\|_{L^{q}(B_R)}^{\frac{q\,p(\s-1)}{N(p-2)+pq}}\\
&\le \, k\,T^{1-\frac{N(\s-1)}{N(p-2)+pq}}\,\delta_1^{\frac{q\,p(\s-1)}{N(p-2)+pq}}\,.
\end{aligned}
\end{equation*}
The thesis follows for $\delta_1>0$ small enough.

\end{proof}

\begin{lemma}\label{lemma6}
Assume \eqref{costanti}, $p>2$ and $s>\max\left\{\s_0,1\right\}.$ Let $u$ be a solution to problem \eqref{eq111} with $u_0\in  L^{\infty}(B_R)$, $u_0\ge0$, such that
\begin{equation}\label{eq66}
\|u_0\|_{L^{s}(B_R)}\le\delta_1,\quad \|u_0\|_{L^{\s\frac Np}(B_R)}\le\delta_1,
\end{equation}
for $\delta_1>0$ sufficiently small. Let $S(t)$ be as in \eqref{eq34}, then
\begin{equation}\label{eq62}
T:=\sup\{t\ge0:\,S(t)\le\,1\}=+\infty.
\end{equation}
\end{lemma}

\begin{proof}
We suppose by contradiction that $$T<+\infty.$$ Then, by \eqref{eq62}, the definition of $S(t)$ in \eqref{eq34} and by Lemma \ref{lemma5} we can write,
\begin{equation}\label{eq67}
\begin{aligned}
1=S(T)&=\sup_{0<t<T}\,t\|u(t)\|_{L^{\infty}(B_R)}^{\s-1}\\
&\le \sup_{0<t<1}\,t\|u(t)\|_{L^{\infty}(B_R)}^{\s-1}+ \sup_{1<t<T}\,t\|u(t)\|_{L^{\infty}(B_R)}^{\s-1}\\
&=:J_1+J_2\,.
\end{aligned}
\end{equation}
Now, by Lemma \ref{lemma3}, applied with $r=s$, and Lemma \ref{lemma71} with $q=s$, we can write
\begin{equation}\label{eq68}
\begin{aligned}
J_1&\le \,\sup_{0<t<1}\,t\left\{k\,t^{-\frac{N}{N(p-2)+ps}}\left(\sup_{\frac t4<\tau<t}\int_{B_R}u^s\,d\mu\right)^{\frac{p}{N(p-2)+ps}}\right\}^{(\s-1)}\\
&\le  \,\sup_{0<t<1}\, k\,t^{1-\frac{N(\s-1)}{N(p-2)+ps}}\left\|u_0\right\|_{L^{s}(B_R)}^{\frac{ps(\s-1)}{N(p-2)+ps}}\,.
\end{aligned}
\end{equation}
On the other hand, for any $q>s$, by Lemma \ref{lemma3}, applied with $r=q$, and Proposition \ref{prop71} with $q_0=s$, we get
\begin{equation}\label{eq69}
\begin{aligned}
J_2&\le \,\sup_{1<t<T}\,t\left\{k\,t^{-\frac{N}{N(p-2)+pq}}\left(\sup_{\frac t4<\tau<t}\int_{B_R}u^q\,d\mu\right)^{\frac{p}{N(p-2)+pq}}\right\}^{(\s-1)}\\
&\le  \,\sup_{1<t<T}\, k\,t^{1-\frac{N(\s-1)}{N(p-2)+pq}}\sup_{\frac t4<\tau<t}\left\|u(\tau)\right\|_{L^{q}(B_R)}^{\frac{pq(\s-1)}{N(p-2)+pq}}\\
&\le  \,\sup_{1<t<T}\, k\,t^{1-\frac{N(\s-1)}{N(p-2)+pq}}\sup_{\frac t4<\tau<t}\left(Ct^{-\frac{s}{p-2}\left(\frac 1{s}-\frac 1q\right)}\left\|u_0\right\|_{L^{s}(B_R)}^{\frac{s}{q}}\right)^{\frac{pq(\s-1)}{N(p-2)+pq}}\\
&\le \,\sup_{1<t<T}\,\frac{C\, k}{4}\,t^{1-\frac{N(\s-1)}{N(p-2)+pq}-\frac{spq(\s-1)}{(p-2)[N(p-2)+pq]}\left(\frac 1{s}-\frac 1q\right)}\left\|u_0\right\|_{L^{s}(B_R)}^{\frac{ps(\s-1)}{N(p-2)+pq}}\,.
\end{aligned}
\end{equation}
By substituting \eqref{eq68} and \eqref{eq69} into \eqref{eq67} we get
\begin{equation}\label{eq610}
1=S(T)\le \sup_{0<t<1}\, k\,t^{a}\left\|u_0\right\|_{L^{s}(B_R)}^{\frac{ps(\s-1)}{N(p-2)+ps}}+\sup_{1<t<T}\,\frac{C\, k}{4}\,t^{b}\left\|u_0\right\|_{L^{s}(B_R)}^{\frac{ps(\s-1)}{N(p-2)+pq}}\,,
\end{equation}
where we have set
$$
a=1-\frac{N(\s-1)}{N(p-2)+ps},\quad\text{and}\quad b=1-\frac{N(\s-1)}{N(p-2)+pq}-\frac{spq(\s-1)}{(p-2)[N(p-2)+pq]}\left(\frac 1{s}-\frac 1q\right)\,.
$$
Now, observe that, since $s>\max\left\{\frac Np(\s-p+1),1\right\}$ and $q>s$,
$$
a>0;\quad\text{and}\quad b<0\,.
$$
Hence, \eqref{eq610}, due to assumption \eqref{eq66}, reads
\begin{equation*}
1=S(T)< k\, \delta_1^{\frac{ps(\s-1)}{N(p-2)+ps}}\,+\,\frac{C\, k}{4}\delta_1^{\frac{ps(\s-1)}{N(p-2)+pq}}\,.
\end{equation*}
Provided that $\delta_1$ is sufficiently small, thus yielding $1=S(T)<1$, a contradiction. Thus $T=+\infty$.
\end{proof}

\begin{proposition}\label{prop4}
Assume \eqref{costanti}, $p>2$ and $s>\max\left\{\s_0,1\right\}$. Let $u$ be a solution to problem \eqref{eq111} with $u_0\in  L^{\infty}(B_R)$, $u_0\ge0$, such that
$$
\|u_0\|_{L^{s}(B_R)}\le\varepsilon_1,\quad \|u_0\|_{L^{\s\frac Np}(B_R)}\le\varepsilon_1,
$$
 with $\varepsilon_1=\varepsilon_1(\s,p,N,C_{s,p},C_p,s)$ sufficiently small. Then, for any $t\in(0,+\infty)$, for some $\Gamma=\Gamma(\s,p,N,q,s,C_{s,p},C_p)>0$
\begin{equation}\label{eq800}
\|u(t)\|_{L^{\infty}(B_R)}\le \Gamma\, t^{-\frac{1}{p-2}\left(1-\frac{ps}{N(p-2)+pq}\right)}\,\|u_0\|_{L^{s}(B_R)}^{\frac{ps}{N(p-2)+pq}}\,.
\end{equation}
\end{proposition}

\begin{proof}
Due to Lemma \ref{lemma6}, $$S(t)\le 1\quad \text{for all}\,\,\,t\in(0,+\infty].$$ Therefore, by Lemma \ref{lemma3} and Proposition \ref{prop71} applied with $q_0=s$, for any $q>s$, we get, for all $t\in (0,+\infty)$
\begin{equation*}
\begin{aligned}
\|u(t)\|_{L^{\infty}(B_R)}&\le \|u\|_{L^{\infty}\left(B_R\times\left(\frac t2,t\right)\right)}\,\\
&\le\, k\,t^{-\frac{N}{N(p-2)+pq}}\left[\sup_{\frac t4<\tau<t}\|u(\tau)\|_{L^q(B_R)}^q\right]^{\frac{p}{N(p-2)+pq}}\\
&\le\,\Gamma\,t^{-\frac{N}{N(p-2)+pq}-\frac{s}{p-2}\left(\frac 1{s}-\frac 1q\right)\frac{pq}{N(p-2)+pq}}\|u_0\|_{L^{s}(B_R)}^{\frac{s}{q}\frac{pq}{N(p-2)+pq}}\,.
\end{aligned}
\end{equation*}
Observing that
\begin{equation*}
-\frac{N}{N(p-2)+pq}-\frac{s}{p-2}\left(\frac 1{s}-\frac 1q\right)\frac{pq}{N(p-2)+pq}=-\frac 1{p-2}\left(1-\frac{ps}{N(p-2)+pq}\right)\,,
\end{equation*}
we get the thesis.
\end{proof}

\begin{proof}[Proof of Theorem \ref{teo71}]
We proceed as in the proof of the previous Theorems. Let $\{u_{0,h}\}_{h\ge 0}$ be a sequence of functions such that
\begin{equation}\label{equ0h}
\begin{aligned}
&(a)\,\,u_{0,h}\in L^{\infty}(M)\cap C_c^{\infty}(M) \,\,\,\text{for all} \,\,h\ge 0, \\
&(b)\,\,u_{0,h}\ge 0 \,\,\,\text{for all} \,\,h\ge 0, \\
&(c)\,\,u_{0, h_1}\leq u_{0, h_2}\,\,\,\text{for any } h_1<h_2,  \\
&(d)\,\,u_{0,h}\longrightarrow u_0 \,\,\, \text{in}\,\, L^{s}(M)\quad \textrm{ as }\, h\to +\infty\,.\\
\end{aligned}
\end{equation}
From standard results it follows that problem \eqref{5} has a solution $u_{h,k}^R$ in the sense of Definition \ref{def31} with $u_{0,h}$ as in \eqref{equ0h}; moreover, $u^R_{h,k}\in C\big([0, \infty); L^q(B_R)\big)$ for any $q>1$.
Due to  Proposition \ref{prop4}, \ref{prop71} and Lemmata \ref{lemma71} and \eqref{lemma6}, the solution $u_{h,k}^R$ to problem \eqref{5} satisfies estimates \eqref{eq72}, \eqref{eq715} and \eqref{eq800} for any $t\in(0,+\infty)$, uniformly w.r.t. $R$, $k$ and $h$.
Thus, by standard arguments, we gan pass to the limit as $R\to+\infty$, $k\to+\infty$ and $h\to+\infty$ and we obtain a solution $u$ to problem \eqref{problema}, which fulfills \eqref{eq801}, \eqref{eq23b} and \eqref{eq22b}.

\end{proof}

\section{Porous medium equation with reaction}\label{porous}

We now consider the following nonlinear reaction-diffusion problem:
\begin{equation}\label{aaa}
\begin{cases}
\, u_t= \Delta u^m +\, u^{\s} & \text{in}\,\, M\times (0,T) \\
\,\; u =u_0 &\text{in}\,\, M\times \{0\}\,,
\end{cases}
\end{equation}
where $M$ is an $N-$dimensional complete noncompact Riemannian manifold of infinite volume, $\Delta$ being the Laplace-Beltrami operator on $M$ and $T\in (0,\infty]$. We shall assume throughout this section that $$ N\geq 3,\quad \quad m\,>\,1,\quad \quad  \s\,>\,m,$$ so that we are concerned with the case of  \it degenerate diffusions \rm of porous medium type (see \cite{V}), and that the initial datum $u_0$ is nonnegative. Let L$^q(M)$ be the space of those measurable functions $f$ such that $|f|^q$ is integrable w.r.t. the Riemannian measure $\mu$. We shall always assume that $M$ supports the Sobolev inequality, namely that:
\begin{equation}\label{aS}
(\textrm{Sobolev\ inequality)}\ \ \ \ \ \ \|v\|_{L^{2^*}(M)} \le \frac{1}{C_s} \|\nabla v\|_{L^2(M)}\quad \text{for any}\,\,\, v\in C_c^{\infty}(M),
\end{equation}
where $C_s$ is a positive constant and $2^*:=\frac{2N}{N-2}$. In one of our main results, we shall also suppose that $M$ supports the Poincar\'e inequality, namely that:
\begin{equation}\label{aP}
(\textrm{Poincar\'e\ inequality)}\ \ \ \ \ \|v\|_{L^2(M)} \le \frac{1}{C_p} \|\nabla v\|_{L^2(M)} \quad \text{for any}\,\,\, v\in C_c^{\infty}(M),
\end{equation}
for some $C_p>0$.

Solutions to \eqref{aaa} will be meant in the very weak, or distributional, sense, according to the following definition.

\begin{definition}\label{adef21}
Let $M$ be a complete noncompact Riemannian manifold of infinite volume, of dimension $N\ge3$. Let $m>1$, $\s>m$ and $u_0\in{\textrm L}^{1}_{\textit{loc}}(M)$, $u_0\ge0$. We say that the function $u$ is a solution to problem \eqref{aaa} in the time interval $[0,T)$ if
$$
u\in L^{\s}_{loc}(M\times(0,T)) 
$$
and for any $\varphi \in C_c^{\infty}(M\times[0,T])$ such that $\varphi(x,T)=0$ for any $x\in M$, $u$ satisfies the equality:
\begin{equation*}
\begin{aligned}
-\int_0^T\int_{M} \,u\,\varphi_t\,d\mu\,dt =&\int_0^T\int_{M} u^m\,\Delta\varphi\,d\mu\,dt\,+ \int_0^T\int_{M} \,u^{\s}\,\varphi\,d\mu\,dt \\
& +\int_{M} \,u_0(x)\,\varphi(x,0)\,d\mu.
\end{aligned}
\end{equation*}
\end{definition}

\medskip

First we consider the case that $\s>m+\frac 2 N$ and the Sobolev inequality holds on $M$. In order to state our results we define
\begin{equation}\label{p0}\s_1:=(\s-m)\frac{N}{2}.\end{equation} Observe that $\s_1>1$ whenever $\s>m+\frac 2N$. We comment that the next results improve and in part correct some of the results of \cite{GMP1}. The proofs are omitted since they are identical to the previous ones.

\begin{theorem}\label{ateo1}
Let $M$ be a complete, noncompact, Riemannian manifold of infinite volume and of dimension $N\ge3$, such that the Sobolev inequality \eqref{aS} holds. Let $m>1$, $\s>m+\frac{2}{N}$, $s>\s_1$ and $u_0\in{\textrm L}^{s}(M)\cap L^1(M)$, $u_0\ge0$.

\begin{itemize}
\item[(i)] Assume that
\begin{equation*}\label{a0}
\|u_0\|_{\textrm L^{s}(M)}\,<\,\varepsilon_0,\quad  \|u_0\|_{\textrm L^{1}(M)}<\,\varepsilon_0\,,
\end{equation*}
with $\varepsilon_0=\varepsilon_0(\s,m,N, C_{s})>0$ sufficiently small.
Then problem \eqref{aaa} admits a solution for any $T>0$,  in the sense of Definition \ref{adef21}. Moreover, for any $\tau>0,$ one has $u\in L^{\infty}(M\times(\tau,+\infty))$ and there exists a constant $\Gamma>0$ such that, one has
\begin{equation*}\label{aeq21tot}
\|u(t)\|_{L^{\infty}(M)}\le \Gamma\, t^{-\alpha}\,\|u_0\|_{L^{1}(M)}^{\frac{2}{N(m-1)+2}}\,\quad\text{ for all $t>0$,}
\end{equation*}
where
$$
\alpha:=\frac{N}{N(m-1)+2}\,.
$$
\item[(ii)] Let  $\s_1\le q<\infty$ 
and
\begin{equation*}\label{a2}
\|u_0\|_{L^{\s_1}(M)}< \hat \varepsilon_0
\end{equation*}
for $\hat\varepsilon_0=\hat\varepsilon_0(\s, m , N, C_s, q)>0$ small enough. Then there exists a constant $C=C(m,\s,N,\varepsilon_0,C_s, q)>0$ such that
\begin{equation*}\label{a3}
\|u(t)\|_{L^q(M)}\le C\,t^{-\gamma_q} \|u_{0}\|^{\delta_q}_{L^{\s_1}(M)}\quad \textrm{for all }\,\, t>0\,,
\end{equation*}
where
$$
\gamma_q=\frac{1}{\s-1}\left[1-\frac{N(\s-m)}{2q}\right],\quad \delta_q=\frac{\s-m}{\s-1}\left[1+\frac{N(m-1)}{2q}\right]\,.
$$
 \item[(iii)] Finally, for any $1<q<\infty$, if $u_0\in {\textrm L}^q(M)\cap\textrm L^{\s_1}(M)$ and
\begin{equation*}\label{a5}
\|u_0\|_{\textrm L^{\s_1}(M)}\,<\,\varepsilon
\end{equation*}
with $\varepsilon=\varepsilon(\s,m,N,r, C_s,q)>0$ sufficiently small, then
\begin{equation*}\label{a6}
\|u(t)\|_{L^q(M)}\le  \|u_{0}\|_{L^q(M)}\quad \textrm{for all }\,\, t>0\,.
\end{equation*}
\end{itemize}
\end{theorem}

\begin{theorem}
Let $M$ be a complete, noncompact manifold of infinite volume and of dimension $N\ge3$, such that the Sobolev inequality \eqref{aS} holds. Let $m>1$, $\s>m+\frac{2}{N}$ and $u_0\in{\textrm L}^{\s_1}(M)$, $u_0\ge0$ where $\s_1$ has been defined in \eqref{p0}.
Assume that
\begin{equation*}\label{a1}
\|u_0\|_{\textrm L^{\s_1}(M)}\,<\,\varepsilon_0
\end{equation*}
with $\varepsilon_0=\varepsilon_0(\s,m,N,r, C_s)>0$ sufficiently small. Then problem \eqref{aaa} admits a solution for any $T>0$,  in the sense of Definition \ref{adef21}. Moreover, for any $\tau>0,$ one has $u\in L^{\infty}(M\times(\tau,+\infty))$ and there exists a constant $\Gamma>0$ such that, one has
\begin{equation*}
\|u(t)\|_{L^{\infty}(M)}\le \Gamma\, t^{-\frac1{\s-1}}\|u_0\|_{L^{\s_1}(M)}^{\frac{\s-m}{\s-1}}\quad \text{for all $t>0$.}
\end{equation*}
Moreover, the statements in (ii) and (iii) of Theorem \ref{ateo1} hold.

\end{theorem}

\medskip

In the next theorem, we address the case that $\s>m$, supposing that both the inequalities \eqref{aS} and \eqref{aP} hold on $M$.

\begin{theorem}
Let $M$ be a complete, noncompact manifold of infinite volume and of dimension $N\ge3$, such that the Sobolev inequality \eqref{aS} and the Poincaré inequality \eqref{aP} hold. Let
$$
m>1,\quad \s>m,
$$
and $u_0\in{\textrm L}^{s}(M)\cap {\textrm L}^{\s\frac N2}(M)$ where $s>\max\left\{1,\s_1\right\}$, $u_0\ge0$.
Assume that
\begin{equation*}\label{a7}
\left\| u_0\right\|_{L^{s}(M)}\,<\,\varepsilon_1, \quad \left\| u_0\right\|_{L^{\s\frac N2}(M)}\,<\,\varepsilon_1,
\end{equation*}
holds with $\varepsilon_1=\varepsilon_1(m,\s,N,r, C_p,C_s)>0$ sufficiently small. Then problem \eqref{aaa} admits a solution for any $T>0$,  in the sense of Definition \ref{adef21}. Moreover for any $\tau>0$ and for any $q>s$ one has $u\in L^{\infty}(M\times(\tau,+\infty))$ and for all $t>0$ one has
\begin{equation*}\label{a8}
\|u(t)\|_{L^{\infty}(B_R)}\le \Gamma\, t^{-\beta_{q,s}}\,\|u_0\|_{L^{s}(B_R)}^{\frac{2s}{N(m-1)+2q}}\,,
\end{equation*}
where
\begin{equation*}\label{a9}
\beta_{q,s}:=\frac{1}{m-1}\left(1-\frac{2s}{N(m-1)+2q}\right)>0\,.
\end{equation*}
Moreover, let $s\le q<\infty$ and
\begin{equation*}\label{a10}
\|u_0\|_{L^{s}(M)}<\hat\varepsilon_1,
\end{equation*}
for some $\hat\varepsilon_1=\hat \varepsilon_1(\s, m ,N, r, C_p, C_s, q,s)>0$ sufficiently small. Then there exists a constant $C=C(\s,m,N,\varepsilon_1,C_s,C_p,q,s)>0$ such that
\begin{equation*}\label{a11}
\|u(t)\|_{L^q(M)}\le Ct^{-\gamma_q} \|u_{0}\|_{L^s(M)}^{\delta_q}\quad \textrm{for all }\,\, t>0\,,
\end{equation*}
where
$$
\gamma_q:=\frac{s}{m-1}\left[\frac 1s-\frac 1q\right],\quad\quad \delta_q:=\frac sq.
$$
Finally, for any $1<q<\infty$, if $u_0\in L^q(M)\cap L^s(M)\cap {\textrm L}^{\s\frac N2}(M)$ and
\begin{equation*}\label{a11}
\|u_0\|_{L^{s}(M)}<\varepsilon,
\end{equation*}
for some $\varepsilon= \varepsilon(\s, m ,N, C_p, C_s, q)>0$ sufficiently small. Then
\begin{equation*}\label{a12}
\|u(t)\|_{L^q(M)}\le  \|u_{0}\|_{L^q(M)}\quad \textrm{for all }\,\, t>0\,.
\end{equation*}
\end{theorem}

%
%

%


\end{document}